\pdfoutput=1
\documentclass[11pt,leqno]{amsart}
\usepackage{epsfig}
\usepackage{amssymb}
\usepackage{amscd}
\usepackage[all,cmtip,matrix,arrow]{xy}
\usepackage{graphicx}
\xyoption{rotate}
\xyoption{pdf}
\usepackage{xparse}

\newtheorem{theo}{Theorem}[section]
\newtheorem{lemma}[theo]{Lemma}
\newtheorem{defi}[theo]{Definition}
\newtheorem{prop}[theo]{Proposition}

\newtheorem{conj}[theo]{Conjecture}
\newtheorem{cor}[theo]{Corollary}
\newtheorem{remark}[theo]{Remark}

\newtheorem{example}[theo]{Example}
\numberwithin{equation}{section}

\mathchardef\mhyphen="2D

\def\A{{\mathbb A}}
\def\bbI{{\mathbb I}}

\def\bL{\mathbb{L}}

\def\Z{\mathbb{Z}}

\def\D{{\mathcal{D}}}

\def\bR{{\mathbf R}}
\def\bL{{\mathbf L}}
\def\bY{{\mathbf Y}}

\def\PP{{\mathbb P}}

\def\pre-tr{\operatorname{pre-tr}}

\def\Hom{\operatorname{Hom}}
\def\End{\operatorname{End}}

\def\Pic{\operatorname{Pic}}

\def\gr{\operatorname{gr}}

\DeclareMathOperator*{\colim}{colim}

\DeclareMathOperator*{\holim}{holim}

\newcommand{\tens}[1]{%
  \mathbin{\mathop{\otimes}\displaylimits_{#1}}%
}

\newcommand{\Ltens}[1]{%
  \mathbin{\mathop{\otimes}\displaylimits^{\bL}_{#1}}%
}

\newcommand{\bbar}{\overline}
\newcommand{\hhat}{\widehat}
\newcommand{\xto}{\xrightarrow}

\newcommand{\hto}{\hookrightarrow}
\newcommand{\leftto}{\leftarrow}
\newcommand{\waveto}{\rightsquigarrow}

\newcommand{\QCoh}{\operatorname{QCoh}}
\newcommand{\Coh}{\operatorname{Coh}}

\newcommand{\bbD}{{\mathbb D}}

\newcommand{\mk}{\mathrm k}

\newcommand{\cF}{{\mathcal F}}
\newcommand{\cG}{{\mathcal G}}
\newcommand{\cO}{{\mathcal O}}

\newcommand{\cL}{{\mathcal L}}
\newcommand{\cM}{{\mathcal M}}
\newcommand{\cN}{{\mathcal N}}
\newcommand{\cD}{{\mathcal D}}

\newcommand{\cA}{{\mathcal A}}
\newcommand{\cB}{{\mathcal B}}
\newcommand{\cI}{{\mathcal I}}
\newcommand{\cC}{{\mathcal C}}
\newcommand{\cE}{{\mathcal E}}

\newcommand{\cS}{{\mathcal S}}
\newcommand{\cT}{{\mathcal T}}
\newcommand{\cK}{{\mathcal K}}
\newcommand{\cH}{{\mathcal H}}

\newcommand{\cHom}{\mathcal Hom}
\newcommand{\mD}{{\mathfrak D}}

\newcommand{\un}{\underline}

\newcommand{\Fun}{\operatorname{Fun}}

\newcommand{\Perf}{\operatorname{Perf}}
\newcommand{\PsPerf}{\operatorname{PsPerf}}
\newcommand{\perf}{\operatorname{perf}}
\newcommand{\pspe}{\operatorname{pspe}}
\newcommand{\propp}{\operatorname{prop}}

\newcommand{\Kar}{\operatorname{Kar}}

\newcommand{\supp}{\operatorname{Supp}}
\newcommand{\Ker}{\operatorname{Ker}}
\newcommand{\coker}{\operatorname{Coker}}

\newcommand{\im}{\operatorname{Im}}

\newcommand{\cha}{\operatorname{char}}
\newcommand{\ch}{\operatorname{ch}}

\newcommand{\rk}{\operatorname{rk}}

\newcommand{\Tor}{\operatorname{Tor}}

\newcommand{\Tr}{\operatorname{Tr}}
\newcommand{\Nm}{\operatorname{Nm}}

\newcommand{\res}{\operatorname{res}}

\newcommand{\Calk}{\operatorname{Calk}}

\newcommand{\Spec}{\operatorname{Spec}}

\newcommand{\Vect}{{\rm Vect}}

\newcommand{\Ho}{\operatorname{Ho}}

\newcommand{\id}{\operatorname{id}}

\newcommand{\dgcat}{\operatorname{dgcat}}

\newcommand{\eval}{\operatorname{ev}}

\newcommand{\tors}{\operatorname{tors}}

\newcommand{\pr}{\operatorname{pr}}

\newcommand{\Mod}{\operatorname{Mod}}

\usepackage{epsf}
\usepackage{amscd}
\newcommand{\m}{\mathfrak{m}}

\title[Categorical formal punctured neighborhood of infinity, I]
{Categorical formal punctured neighborhood of infinity, I}

\author{Alexander I. Efimov}
\address{Steklov Mathematical Institute of RAS, Gubkin str. 8, GSP-1, Moscow 119991, Russia\\
National Research University Higher School of Economics, Russian Federation}
\email{efimov@mccme.ru}
\thanks{The author is partially supported by Laboratory of Mirror Symmetry NRU HSE, RF government  grant, ag. N 14.641.31.0001}

\begin{document}

\begin{abstract}In this paper we introduce and study the formal punctured neighborhood of infinity, both in the algebro-geometric and in the DG categorical frameworks. For a smooth algebraic variety $X$ over a field of characteristic zero, one can take its smooth compactification $\bbar{X}\supset X,$ and then take the DG category of perfect complexes on the formal punctured neighborhood of the infinity locus $\bbar{X}-X.$ The result turns out to be independent of $\bbar{X}$ (up to a quasi-equivalence) and we denote this DG category by $\Perf(\hhat{X}_{\infty}).$

According to Mohammed Abouzaid \cite{Ab}, there is an analogue of this construction in the framework of Fukaya categories, which uses infinite-dimensional local systems.

We show that this construction can be done purely DG categorically (hence of course also $A_{\infty}$-categorically). For any smooth DG category $\cB,$ we construct the DG category $\Perf_{top}(\hhat{\cB}_{\infty}),$ which we call the category of perfect complexes on the formal punctured neighborhood of infinity of $\cB.$ 
The construction is closely related to the algebraic version of a Calkin algebra: endomorphisms of an infinite-dimensional vector space modulo endomorphisms of finite rank. We prove that for $X$ is as above and $\cB=\Perf(X)$ one obtains exactly the category $\Perf(\hhat{X}_{\infty}).$ 

We study numerous examples. In particular, for the algebra of rational functions on a smooth complete connected curve $C$ we obtain the algebra of adeles $\A_C,$ and for $\cB=D^b_{coh}(Y)$ for a proper singular scheme $Y$ we obtain the category $D_{sg}(Y)^{op}$ -- the opposite category of the Orlov's category of singularities. Among other things, we discuss the relation with the papers of Tate \cite{Ta} and Arbarello, de Concini, and Kac \cite{ACK}.   
\end{abstract}


\maketitle

\tableofcontents

\section{Introduction}

In this paper we define the DG categorical version of formal punctured neighborhood of infinity. The algebro-geometric motivation is the following. Let $X$ be a smooth algebraic variety over a field $\mk$ of characteristic zero. Choosing a smooth compactification $\bbar{X}\supset X$ with $Z=\bbar{X}-X,$ we can take the formal neighborhood $\hhat{X}_Z$ and then take away $Z,$ i.e. consider the complement $\hhat{X}_Z-Z.$ This object lacks a rigorous mathematical framework, but at least on the level of triangulated DG categories we can define the category of perfect complexes $\Perf(\hhat{X}_Z-Z)$ to be the (homotopy) Karoubi envelope of the quotient $\Perf(\hhat{X}_Z)/\Perf_Z(\hhat{X}_Z).$ The result in fact does not depend on the choice of a compactification (up to a quasi-equivalence, see Corollary \ref{cor:well_defined_V_infty}), and is therefore an invariant of $X.$ Let us denote this category by $\Perf(\hhat{X}_{\infty}).$

According to Mohammed Abouzaid \cite{Ab}, there is an analogue of this construction in symplectic geometry, in the framework of Fukaya categories. This motivates a natural question: can we do it purely categorical terms? In other words, can we define for any (at least smooth) DG category its "categorical formal punctured neighborhood of infinity" which for $\Perf(X)$ would give exactly $\Perf(\hhat{X}_{\infty})$?

It turns out that the answer is "yes": for any smooth DG category $\cB$ there is a DG category which we denote by $\Perf_{top}(\hhat{\cB}_{\infty}),$ with a natural functor $\Perf(\cB)\to \Perf_{top}(\hhat{\cB}_{\infty}).$ The image of this functor is denoted by $\hhat{\cB}_{\infty}.$ Moreover, the DG category $\hhat{\cB}_{\infty}$ (unlike $\Perf_{top}(\hhat{\cB}_{\infty})$) can be defined for an arbitrary small DG category $\cB.$

To illustrate the construction we first describe its "non-derived" version (see Section \ref{sec:assoc_algebras}). Namely suppose that $B$ is an associative algebra, considered as a DG algebra concentrated in degree zero, with $d=0.$ Then $\hhat{B}_{\infty}$ is a DG algebra with non-negative cohomology, and the algebra $H^0(\hhat{B}_{\infty})$ is described as follows (Proposition \ref{prop:descr_of_H^0_neighborhood}):
\begin{equation}
\label{eq:non-derived_version_intro}H^0(\hhat{B}_{\infty})\cong\{\varphi\in\End_{\mk}(B)\mid \forall b\in B\, \rk [\varphi,R_b]<\infty\}/B^*\otimes B.\end{equation}
Here $R_b:B\to B$ denotes the operator of right multiplication by $B,$ and $B^*\otimes B$ is identified with the space of linear endomorphisms of $B$ of finite rank. The homomorphism $B\to H^0(\hhat{B}_{\infty})$ is given by $b\mapsto \bbar{L_b},$ where $L_b$ is the operator of left multiplication by $b.$ It is a pleasant exercise to verify that for $B=\mk[t]$ the RHS of \eqref{eq:non-derived_version_intro} is naturally isomorphic to $\mk((t^{-1})),$ see Example \ref{ex:computation_for_k[t]}. 

For a general construction we need to define the Calkin DG category $\Calk_{\mk}.$ It is a DG quotient of $\Mod_{\mk}$ by $\mk$ (or, equivalently, by perfect complexes $\Perf(\mk)$). More precisely, its objects are complexes of $\mk$-modules, and $\Calk_{\mk}(V,W)=\Hom_{\mk}(V,W)/V^*\otimes W.$

\begin{remark}\label{remark:original_Calkin} Originally \cite{C}, Calkin algebra is defined for a separable infinite-dimensional Hilbert space $H$ to be the quotient $B(H)/C(H)$ of bounded operators on $H$ by the two-sided ideal of compact operators. Since bounded operators on $H$ of finite rank are dense in $C(H),$ it is natural to define the "algebraic" Calkin algebra of an infinite-dimensional vector space $V$ to be $\End_{\mk}(V)/V^*\otimes V.$ The DG category $\Calk_{\mk}$ is then a straightforward generalization. A non-DG version has also been considered e.g. in \cite{Dr}.\end{remark}

For a DG algebra $B$ we have $\hhat{B}_{\infty}:=C^{\bullet}(B^{op},\End_{\Calk_{\mk}}(B)).$ The multiplication on $\hhat{B}_{\infty}$ comes from a homomorphism $B^{op}\to \End_{\Calk_{\mk}}(B$). 
The morphism of DG algebras $B\to \hhat{B}_{\infty}$ is given by the composition $$B\xto{\sim} C(B^{op},\End_{\mk}(B))\to C^{\bullet}(B^{op},\End{\Calk_{\mk}}(B)).$$ A similar construction works for an arbitrary small DG category, and it is Morita invariant (see Section \ref{sec:DG_cat_neighborhoods_main_result}).

The definition of $\Perf_{top}(\hhat{\cB}_{\infty})$ is more tricky, it is obtained in Section \ref{sec:DG_cat_neighborhoods_main_result}. Here we just say that $\Perf_{top}(\hhat{\cB}_{\infty})$ is a certain full DG subcategory of the category of "almost DG $\cB$-modules" $\bR\un{\Hom}(B^{op},\Calk_{\mk})$ (this has  nothing to do with almost mathematics developed by Faltings \cite{F} and Gabber-Romero \cite{GaRom}). We also call the objects of $\Perf_{top}(\hhat{\cB}_{\infty})$ "perfect almost DG modules".  We refer the reader to Section \ref{sec:DG_cat_neighborhoods_main_result} for the precise definition and for justification of the terminology. 



Our first result concerns a general noetherian scheme $X$ and a closed subscheme $Z.$ The category $\Perf(\hhat{X}_Z)$ is the homotopy limit of DG categories $\Perf(Z_n),$ where $Z_n\subset X$ is the $n$-th infinitesimal neighborhood of $Z.$ Such a homotopy limit looks (and in some sense is) hard to deal with, but it turns out that this category has a useful alternative description. For simplicity, in the introduction we tacitly identify triangulated categories with their DG enhancements.

We prove that the category $\Perf(\hhat{X}_Z)$ is equivalent to a full subcategory $\cT_Z$ of $D_Z(\QCoh(X))$ (Theorem \ref{th:perf_of_formal_is_T_Z}). Namely, $\cT_Z$ consists of objects $\cF$ such that for any perfect complex $\cG\in\Perf_Z(T),$ supported on $Z,$ we have $\cG\Ltens{\cO_X}\cF\in\Perf_Z(X)$ (equivalently, replacing $\cG$ by $\cG^{\vee},$ we require that $\bR\cHom_{\cO_X}(\cG,\cF)\in \Perf_Z(X)$). The functor $\cT_Z\xto{\sim}\Perf(\hhat{X}_Z)$ is given by pullback, and the quasi-inverse $\Perf(\hhat{X}_Z)\xto{\sim}\cT_Z$ is given, roughly speaking, by composing pushforward from $\hhat{X}_Z$ to $X$ and the functor of sections supported on $Z,$ denoted by $\cH_Z:D(X)\to D_Z(X).$ For details see Section \ref{sec:formal_neighborhood}.

The important application of this description of perfect complexes on the formal neighborhood is the following result.

\begin{theo}\label{th:intro_neighborhoods_and_D_sg} Let $X$ be smooth and proper scheme over $k,$ and $Z\subset X$ a closed subscheme. Let us put $\cS:=\Perf_Z(X).$ Then we have a natural equivalence $\Perf(\hhat{X}_Z)\simeq \PsPerf(\cS),$ compatible with the inclusions of $\Perf_Z(X).$ Passing to (Karoubi envelopes of) the quotients, we get an equivalence
$\Perf(\hhat{X}_Z-Z)\cong \bbar{D_{sg}}(\cS).$\end{theo}

Here $\PsPerf(\cS)$ denotes the category of pseudo-perfect $\cS$-modules, and $\bbar{D_{sg}}(\cS)$ is the Karoubi envelope of the category
$$D_{sg}(\cS):=\PsPerf(\cS)/\Perf(\cS).$$ Our main result (Theorem \ref{th:main_theorem}) states that whenever we have a short exact sequence of DG categories $\cS\hto\cA\to\cB$ with $\cA$ being smooth and proper (hence $\cS$ proper and $\cB$ smooth), we have a natural equivalence $\bbar{D_{sg}}(\cS)\simeq \Perf_{top}(\hhat{\cB}_{\infty}).$

In the algebro-geometric setting (for general noetherian schemes) an object of $\Perf(\hhat{X}_Z)$ is called {\it algebraizable} if it is contained in the thick triangulated subcategory generated by the image of the restriction functor 
$\Perf(X)\to\Perf(\hhat{X}_Z).$ The same terminology applies to the formal punctured neighborhood 
$\hhat{X}_Z-Z.$ Similarly, for a smooth DG category $\cB$ an object of 
$\Perf_{top}(\hhat{\cB}_{\infty})$ is called algebraizable if it can be generated by the image of 
$\Perf(\cB)\to \Perf_{top}(\hhat{\cB}_{\infty}).$ The corresponding full subcategories are denoted by $\Perf_{alg}(\hhat{X}_Z),$ $\Perf(\hhat{X}_Z-Z)$ and $\Perf_{alg}(\hhat{\cB}_{\infty}).$

\begin{remark}\label{rem:algebraizable_not_extendable}Note that algebraizability condition for a perfect complex on $\hhat{X}_Z$ is much weaker than extendability to $X,$ and similarly for $\hhat{X}_Z-Z$ and $\hhat{\cB}_{\infty}.$ For example, it is shown in Appendix \ref{app:formal_affine} that if $X$ is affine, then any perfect complex on $\hhat{X}_Z$ is algebraizable (Proposition \ref{prop:completion_affine}).\end{remark}  

In Sections \ref{sec:assoc_algebras}-\ref{sec:neighborhood_for_D^b_coh} we consider a number of examples of our construction. In particular, in Section \ref{sec:affine_plane} we apply our general results to line bundles on $\hhat{(\A_{\mk}^2)}_{\infty}.$ In this case the Picard group is large and is identified with a certain multiplicative group of formal power series (Corollary \ref{cor:Pic_punctured_neighborhood}). We obtain an interesting relation between algebraizability of a line bundle (as a perfect complex) and algebraicity of the power series (Theorem \ref{th:algebraizable_algebraic}). 

A somewhat surprising example of our construction is the following result (Theorem \ref{th:neighborhood_of_infty_for_D^b_coh}).

\begin{theo}\label{th:intro_neighborhood_of_infty_for_D^b_coh}Let $X$ be a proper scheme over a perfect field $\mk.$ Then we have a natural quasi-equivalence $D_{sg}(X)^{op}\xto{\sim} \hhat{D^b_{coh}(X)}_{\infty}.$\end{theo}

Here $D_{sg}(X)^{op}$ is the opposite category of the Orlov's category of singularities
$D^b_{coh}(X)/\Perf(X).$

The paper is organized as follows.

Section \ref{sec:preliminaries} contains mostly preliminaries on DG categories, and some notation. Here we define the Calkin DG category $\Calk_{\cC}$ of a DG category $\cC,$ and prove that a short exact sequence of DG categories induces a short exact sequence of the associated Calkin DG categories (Proposition \ref{prop:short_exact_Calkins}).

In Section \ref{sec:formal_neighborhood} we define the category of perfect complexes on the formal neighborhood of a closed subscheme $Z\subset X,$ with $X$ being noetherian separated, and prove Theorem \ref{th:perf_of_formal_is_T_Z} which states an equivalence between $\Perf(\hhat{X}_Z)$ and the category $\cT_Z$ defined above. The proof is technically involved, and occupies essentially the whole section. In Subsection \ref{ssec:X_smooth_proper} we deduce Theorem \ref{th:intro_neighborhoods_and_D_sg}.

In Section \ref{sec:DG_cat_neighborhoods_main_result} we give the construction of $\Perf_{top}(\hhat{\cB}_{\infty}).$ The main result of this section is Theorem \ref{th:main_theorem}. 

In Section \ref{sec:generalities_on_the_construction} we list some general properties of our construction $\cB\waveto \Perf_{top}(\hhat{\cB_{\infty}}),$ in particular, the kernel of the functor $\bbar{\bY}^*:\Perf(\cB)\to\Perf_{top}(\hhat{\cB}_{\infty})$ is identified with $\PsPerf(\cB)$ (Proposition \ref{prop:pseudo-perfect_in_the_kernel}). 

In section \ref{sec:assoc_algebras} we obtain the non-derived version of our construction given by  \eqref{eq:non-derived_version_intro} (Proposition \ref{prop:descr_of_H^0_neighborhood}). We also show that for formally smooth associative algebras (Definition \ref{def:Quillen_smooth}) the non-derived version coincides with the actual one (\ref{prop:Quillen_smooth}). 

In section \ref{sec:smooth_affine_variety} we describe the cohomology algebra $H^{\bullet}(\hhat{B}_{\infty}),$ where $\Spec B$ is a smooth affine variety of dimension $d\geq 2$ (Proposition \ref{prop:smooth_affine_var_neighborhood_infty}). 

Section \ref{sec:affine_plane} is devoted to detailed study of the case of affine plane $\A_{\mk}^2=\Spec\mk[x,y],$ and in particular of line bundles on $\hhat{(\A_{\mk}^2)}_{\infty}$. The Picard group is identified with the multiplicative group of power series in $x^{-1},y^{-1},$ of the form $1+x^{-1}y^{-1}f(x^{-1},y^{-1})$  (Proposition \ref{prop:Pic_formal_neighborhood}, Corollary \ref{cor:Pic_punctured_neighborhood}). We describe the almost $\mk[x,y]$-modules corresponding to the line bundles (\ref{prop:M_g_and_L_g}). For $g$ of the form $1+x^{-1}y^{-1}h(y^{-1})$ we prove that algebraizability of $L_g$ as a perfect complex is equivalent to the algebraicity of the series $h$ (Theorem \ref{th:algebraizable_algebraic}). We expect such equivalence to hold for arbitrary $g.$

In Section \ref{sec:neighborhood_for_D^b_coh} we consider the case of the derived category of coherent sheaves on a proper (but not smooth) scheme and prove Theorem \ref{th:intro_neighborhood_of_infty_for_D^b_coh}.

In Section \ref{sec:concluding} we discuss (without proofs) some aspects of our construction which were not covered in the present paper. In Subsection \ref{ssec:residues_etc} we explain the relation with the papers of Tate \cite{Ta} and Arbarello, de Concini and Kac \cite{ACK}. In Subsection \ref{ssec:more_invariants} we sketch a general principle which allows to obtain invariants of locally proper DG categories and smooth DG categories, that are related to each other as in Theorem \ref{th:main_theorem}. We discuss various interesting examples.

In Appendix \ref{app:formal_affine} we show that for an affine scheme $X=\Spec A,$ and a closed subscheme $Z\subset X$ corresponding to an ideal $I\subset A,$ the category $\Perf(\hat{X}_Z)$ is nothing else but $\Perf(\hhat{A}_I)$ (in other words, all perfect complexes on $\hhat{X}_Z$ are algebraizable). 

In Appendix \ref{app:boundedness_etc} we prove various technical statements related to compact approximation, boundedness of DG modules and pseudo-perfect DG modules.

{\noindent{\bf Acknowledgements.}} I am grateful to Mohammed Abouzaid, Alexander Beilinson, Dmitry Kaledin, Maxim Kontsevich, Dmitri Orlov, Paul Seidel, Carlos Simpson and Yan Soibelman for useful discussions. I am especially grateful to Mohammed Abouzaid for asking me a question if (the category of perfect complexes on) the algebro-geometric formal punctured neighborhood of infinity can be obtained DG categorically.




\section{Preliminaries on DG categories}
\label{sec:preliminaries}

For the introduction on DG categories, we refer the reader to \cite{Ke}. The references for DG quotients are \cite{Dr2, Ke2}. For the model structures on DG categories we refer the reader to \cite{Tab1, Tab2}, and for a general introduction on model categories we refer to \cite{Ho}. To simplify the exposition we do not discuss set-theoretic issues, referring to \cite{T, TV}. All statements and constructions will be done over some base field $\mk,$ although almost everything can be done over an arbitrary base commutative ring (for some statements one needs $\mk$ to be noetherian or regular noetherian). We will specify when we need $\mk$ to be perfect or to have characteristic zero. 

All modules are assumed to be right unless otherwise stated. For a small DG category $\cC$ and a $\cC$-module $M,$ we denote by $M^{\vee}$ the $\cC^{op}$-module $\Hom_{\cC}(M,\cC).$ We denote by $M^*$ the $\cC^{op}$-module $\Hom_{\cC}(M,\mk).$

For a small DG category $\cC,$ we denote by $\Mod_{\cC}$ the DG category of cofibrant DG $\cC$-modules in the projective model structure (these are exactly the direct summands of semifree DG $\cC$-modules).
We denote by $\bY:\cC\hto\Mod_{\cC}$ the standard Yoneda embedding given by $\bY(x)=\cC(-,x).$

The diagonal $\cC\mhyphen\cC$-bimodule is denoted by $I_{\cC}$ or just $\cC$ if this does not lead to confusion. We denote by $\cC^{!}\in D(\cC\otimes\cC^{op})$ the bimodule $\bR\Hom_{\cC\otimes\cC^{op}}(\cC,\cC_{\cC}\otimes {}_{\cC}\cC),$ where the subscript ${}_{\cC}$ is inserted to clarify which $\cC$-action (left or right) we consider.

The Calkin DG category $\Calk_{\cC}$ has the same objects as $\Mod_{\cC},$ and the morphisms are given by
\begin{equation}\label{eq:morphisms_in_Calk_C}\Calk_{\cC}(M,N):=\coker(N\otimes_{\cC}M^{\vee}\xto{\eval}\Hom_{\cC}(M,N)).\end{equation} 

\begin{prop}\label{prop:Calkin_DG_quotient}The natural functor $\Mod_{\cC}/\bY(\cC)\to \Calk_{\cC}$ is a quasi-equivalence.\end{prop}

\begin{proof}Recall that in a DG quotient $\cA/\cB$ the morphisms are given by the complexes $$(\cA/\cB)(x,y)=Cone(\cA(-,y)\Ltens{\cB}\cA(x,-)\to\cA(x,y)),$$
where the derived tensor product is computed via bar resolution. In our case the naive tensor product is quasi-isomorphic to the derived one, and the evaluation morphism $\eval$ in \eqref{eq:morphisms_in_Calk_C} is injective. This proves the assertion.\end{proof}

Note that for $\cC=\mk,$ we get the DG category $\Calk_{\mk}$ which has been already discussed in the introduction. It will be convenient for us to introduce the homotopy Karoubi envelope $\bbar{\Calk}_{\cC}\supset\Calk_{\cC}.$ For a DG functor $\Phi:\cC_1\to\cC_2$ we denote by $\Calk_{\Phi}:\Calk_{\cC_1}\to \Calk_{\cC_2}$ the induced (extension of scalars) functor, and similarly for $\bbar{\Calk}_{\Phi}.$

\begin{remark}We note that $K_0(\bbar{\Calk}_{\cC})=K_{-1}(\Perf(\cC)),$ hence we have $\Calk_{\cC}=\bbar{\Calk}_{\cC}$ if and only if $K_{-1}(\Perf(\cC))=0.$ In particular, we have $\Calk_{\mk}=\bbar{\Calk}_{\mk}$ (since we assume $\mk$ to be a field).\end{remark}

Recall that a $\cC$-module $M$ is pseudo-perfect if for each $x\in\cC,$ the complex $M(x)$ is perfect over $\mk$ \cite{TV}. We write $\Perf(\cC)\subset \Mod_{\cC}$ (resp. $\PsPerf(\cC)\subset\Mod_{\cC}$) for the full subcategory of perfect (resp. pseudo-perfect) $\cC$-modules.

For a DG category $T,$ we denote by $[T]$ its (non-graded) homotopy category, which has the same objects as $T,$ and the morphisms are given by $[T](X,Y)=H^0(T(X,Y)).$ We use the the terminology of \cite[Definition 2.4]{TV} by calling a DG category $\cC$ triangulated if the Yoneda embedding provides a quasi-equivalence $\cC\xto{\sim}\Perf(\cC).$ In this case $[\cC]$ is a Karoubi complete triangulated category. We have triangulated categories $D(\cC)\simeq [\Mod_{\cC}],$ $D_{\perf}(\cC)\simeq [\Perf(\cC)],$ $D_{\pspe}(\cC)\simeq [\PsPerf(\cC)].$ 

We write $T^{\Kar}$ or $\bbar{T}$ for the homotopy Karoubi completion of a DG category $T;$ for example, $T^{\Kar}$ can be defined as a full DG subcategory of $\Perf(T),$ corresponding to the idempotent completion of $[T]\subset D_{\perf}(T).$

For a DG functor $\Phi:\cC_1\to\cC_2,$ we have an extension of scalars DG functor $\Phi^*:\Mod_{\cC_1}\to\Mod_{\cC_2},$ which have a right adjoint quasi-functor $\Phi_*:\Mod_{\cC_2}\to\Mod_{\cC_1}.$ We also denote by $\bL\Phi^*$ and $\Phi_*$ the corresponding exact functors between $D(\cC_1)$ and $D(\cC_2).$ 

We also recall from \cite[Definitions 3.6]{T} that a $\cC$-module is called quasi-representable if it is quasi-isomorphic to a representable $\cC$-module. For two DG categories $\cC,\cC',$ a $\cC\otimes\cC'$-module $M$ is called right quasi-representable if for each object $X\in\cC,$ the $\cC'$-module $M(X,-)$ is quasi-representable.

We denote by $\bR\un{\Hom}(\cC,\cC')\subset\Mod_{\cC^{op}\otimes\cC'}$ the full subcategory of right quasi-representable $\cC^{op}\otimes\cC'$-modules. By \cite[Theorem 6.1]{T}, this DG category (considered up to a quasi-equivalence) is actually the internal Hom in the homotopy category of DG categories $\Ho(\dgcat_{\mk})$ (with inverted quasi-equivalences). We have a natural quasi-functor $\Fun(\cC,\cC')\to \bR\un{Hom}(\cC,\cC'),$ where $\Fun(\cC,\cC')$ is the naive DG category of DG functors $\cC\to\cC',$ as defined in \cite{Ke}. Moreover, if $\cC$ is cofibrant, this functor is essentially surjective on the homotopy categories.

A small DG category $\cC$ is called smooth (resp. locally proper) if the diagonal $\cC\mhyphen\cC$-bimodule is perfect (resp. pseudo-perfect). Moreover, $\cC$ is called proper if it is locally proper and is Morita equivalent to a DG algebra (i.e. the triangulated category $D_{\perf}(\cC)$) has a classical generator). The following (simple but useful) criterion of smoothness and local properness will be important:

$\bullet$ A small DG category $\cC$ is smooth (resp. locally proper) if for any small DG category $\cC'$ we have an inclusion $\bR\un{\Hom}(\cC,\cC')\subset \Perf(\cC^{op}\otimes\cC')$ (resp. $\Perf(\cC^{op}\otimes\cC')\subset \bR\un{\Hom}(\cC,\cC')$).

We recall the notion of a short exact sequence of DG categories.

\begin{defi}\label{defi:short_exact_dg}A pair of functors $\cA_1\xto{F_1}\cA_2\xto{F_3}\cA_3$ is said to be a (Morita) short exact sequence of DG categories if the following conditions hold

$\rm{i)}$ the composition $F_2F_1$ is homotopic to zero;

$\rm ii)$ the functor $F_1$ is quasi-fully-faithful;

$\rm iii)$ the induced quasi-functor $\bbar{F_2}:\cA_2/F_1(\cA_1)\to \cA_3$ is a Morita equivalence.
\end{defi}

For a DG functor (or quasi-functor) $F:\cC_1\to\cC_2$ its kernel $\Ker(F)\subset\cC_1$ is the full subcategory formed by objects $x$ such that $[F](x)=0\in[\cC_2].$

\begin{prop}\label{prop:short_exact_Calkins}For a short exact sequence of DG categories $\cA_1\xto{F_1}\cA_2\xto{F_2} \cA_3,$ we have a short exact sequence $\Calk_{\cA_1}\xto{\Calk_{F_1}}\Calk_{\cA_2}\xto{\Calk_{F_2}}\Calk_{\cA_3}.$ In particular, we have a quasi-equivalence $\bbar{\Calk}_{\cA_1}\simeq\Ker(\bbar{\Calk}_{F_2}).$\end{prop}

\begin{proof}The only assertion that needs clarification here is the quasi-fully-faithfulness of $\Calk_{F_1}.$ Indeed, the rest would follow from the exactness of $\Mod_{\cA_1}\xto{F_1^*}\Mod_{\cA_2}\xto{F_2^*}\Mod_{\cA_3}.$

Now, take two objects $M,N\in \Mod_{\cA_1}.$ It suffices to prove that the natural morphism $N\tens{\cA_1}M^{\vee}\to F_1^*(N)\tens{\cA_2} F_1^*(M)^{\vee}$ is a quasi-isomorphism. This in turn follows from the quasi-isomorphism $M^{\vee}\xto{\sim} F_{1*}(F_1^*(M))$ 
(following from quasi-fully-faithfulness of $F_1$) and the bifunctorial isomorphism 
$\bL F_1^*(-)\Ltens{\cA_2}?\cong -\Ltens{\cA_1}F_{1*}(?).$ The proposition is proved.
\end{proof}

Finally, we make some comments on our usage of terminology below. Mostly we will consider DG categories up to a quasi-equivalence. By a functor between DG categories we sometimes mean a quasi-functor. In some cases it is convenient for us to choose a concrete DG model or a concrete DG functor. By a commutative diagram of functors we usually mean the commutative diagram in the homotopy category $\Ho(\dgcat_{\mk}).$ Finally, we denote by $\Ho_M(\dgcat_{\mk})$ the Morita homotopy category of DG categories (with inverted Morita equivalences).

\section{Description of perfect complexes on a formal neighborhood $\hhat{X}_Z$}
\label{sec:formal_neighborhood}

For a noetherian separated scheme $Y$ we denote by $\mD(Y)$ the DG enhancement of the derived category of quasi-coherent sheaves $D(Y):=D(\QCoh Y),$ given by the quotient of the homotopy category of h-flat complexes by the full subcategory of acyclic h-flat complexes, see e.g. \cite[Section 3.10]{KL}. We denote by $\Perf(Y)\subset \mD(Y)$ (resp. $\mD^b_{coh}(Y)\subset \mD(Y)$) the full DG subcategory of perfect complexes (resp. complexes with bounded coherent cohomology). Also, for a closed subset $S\subset Y$ we denote by $\mD_{S}(Y)$ the full DG subcategory of complexes whose cohomology supported on $S.$ Clearly, $\Perf(Y)$ (resp. $\mD^b_{coh}(Y),$ $\mD_{S}(Y)$) is an enhancement of $D_{\perf}(Y)$ (resp. $D^b_{coh}(Y),$ $D_S(Y)$). We also put $\Perf_S(Y):=\Perf(Y)\cap \mD_S(Y)$ (resp. $D_{\perf,S}(Y):=D_{\perf}(Y)\cap D_S(Y)$). We denote by $\cH_S:D(Y)\to D_S(Y)$ the right adjoint to the inclusion functor. 

These enhancements are convenient for us because of compatibility with pullbacks: a morphism $f:Y\to Y'$ induces an actual DG functor $f^*:\mD_(Y')\to\mD_(Y)$ (compatible with the derived pullback functor $\bL f^*$) which takes $\Perf(Y')$ to $\Perf(Y).$

Let now $X$ be a noetherian separated scheme over $\mk,$ and $Z\subset X$ a closed subscheme. Let us denote by $Z_n\subset X$ the $n$-th infinitesimal neighborhood of $Z$ in $X,$ defined by the sheaf of ideals $\cI_Z^n\subset\cO_X.$ We denote by $\iota_n:Z_n\to X$ and $\iota_{k,n}:Z_k\to Z_n$ ($k<n$) the inclusion morphisms.

\begin{defi}\label{defi:perfect_complexes_formal} The DG category $\Perf(\hhat{X}_Z)\in\dgcat_k$ of (non-algebraizable) perfect complexes on the formal scheme $\hhat{X}_Z$ is defined as the homotopy limit $\holim\limits_n \Perf(Z_n).$We put $D_{\perf}(\hhat{X}_Z):=[\Perf(\hhat{X}_Z)].$\end{defi}

\begin{remark}\label{rem:remark_on_perfect_on_formal}This is natural definition of a (DG) category of perfect complexes on the formal neighborhood, and it can be generalized straightforwardly to an abstract formal scheme. Our definition is compatible with the definition of Gaitsgory and Rozenblyum for ind-schemes \cite{GR}.\end{remark}

We define the full triangulated subcategory $T_Z\subset D_Z(X)$ by the following condition:
\begin{equation}
\label{eq:definition_of_T_Z}T_Z=\{\cF\in D_Z(X)\mid \text{for any }\cG\in D_{\perf,Z}(X)\text{ we have }\cG\Ltens{\cO_X}\cF\in D_{perf,Z}(X)\}.
\end{equation}
Note that equivalently $\cF\in T_Z$ iff $\bR\cHom_{\cO_X}(\cG,\cF)\in D_{\perf,Z}(X)$ for any $\cG\in D_{\perf,Z}(X).$
We denote by $\cT_Z\subset\mD_Z(X)$ the corresponding full DG subcategory.

The main result of this section is the following theorem.

\begin{theo}
\label{th:perf_of_formal_is_T_Z}There is a natural quasi-equivalence $\cT_Z\xto{\sim}\Perf(\hhat{X}_Z),$ given by pullback. We have a commutative triangle
\begin{equation}\label{eq:compatibility_restrictions_geometric}
\xymatrix{\Perf(X)\ar[rd]\ar[r]^{\cH_Z} & \cT_Z\ar[d]^-[@!-90]{\sim}\\
& \Perf(\hhat{X}_Z),
}
\end{equation}
where the diagonal arrow is the pullback functor.
\end{theo}

We will need a series of technical results. Recall that for any collection of functions $f_1,\dots,f_m\in\cO(X)$ the associated  Koszul complex is defined by the formula
\begin{equation}
\label{eq:Koszul_complex}\cK(f_1,\dots,f_m)=\bigotimes\limits_{i=1}^m \{\cO_X\xto{f_i}\cO_X\},
\end{equation}
where the tensor product is over $\cO_X,$ and each of the two-term complexes is in degrees $-1,0.$

\begin{lemma}\label{lem:restriction_is_perfect}For any $\cF\in T_Z,$ and any nilpotent thickening $Z\subset Z'\subset X$ we have 
$\bL\iota_{Z'}^*\cF\in D_{\perf}(Z').$\end{lemma}

\begin{proof}The statement is local, so we may and will assume that $X$ is affine. Choose a finite sequence of functions $f_1,\dots,f_m\in\cO(X)$ which generate the ideal $\Gamma(X,\cI_{Z'})\subset\cO(X).$ Then we have $\cK(f_1,\dots,f_m)\in D_{\perf,Z}(X).$ Since $\cF$ is in $T_Z,$ we also have $\cK(f_1,\dots,f_m)\tens{\cO_X}\cF\in D_{\perf,Z}(X).$ Since perfect complexes are preserved by pullbacks, we have $\bL\iota_{Z'}^*(\cK(f_1,\dots,f_m)\tens{\cO_X}\cF)\in D_{\perf}(Z').$ But on the other hand, since each of the $f_i$ vanishes on $Z',$ we have $$\bL\iota_{Z`}^*(\cK(f_1,\dots,f_m)\tens{\cO_X}\cF)\cong\bigoplus_{j=0}^m \bL\iota_{Z'}^*(\cF)^{\oplus\binom{m}{j}}[j].$$ This implies that $\bL\iota_{Z'}^*(\cF)$ is a direct summand of $\bL\iota_{Z'}^*(\cK(f_1,\dots,f_m)\tens{\cO_X}\cF),$ hence it is also a perfect complex. This proves the lemma.\end{proof}

\begin{lemma}\label{lem:fully_faithful_for_perfect_affine} Suppose that $X$ is affine, and $\cG\in D_{\perf,Z}(X).$ Then the natural map
$\cG\to\holim\limits_{n}\iota_{n*}\bL\iota_n^*(\cG)$ is an isomorphism in $D(X).$\end{lemma}

\begin{proof}Let $X= \Spec A,$ and denote by $I\subset A$ the ideal defining $Z.$ It is classically known \cite{AM} that the completion $\hhat{A}_I$ is a flat $A$-module, and for any finitely generated $A$-module $M$ we have a natural isomorphism $M\otimes \hhat{A}_I\cong \lim\limits_{n} M/I^n M.$ It follows that for any complex $M^{\bullet}\in D^b_{f.g.,I-tors}(A)$ (with bounded finitely generated $I$-torsion cohomology) we have an isomorphism 
\begin{equation}\label{eq:tensoring_torsion_with_completion}
M^{\bullet}\xto{\sim}M^{\bullet}\Ltens{A}\hhat{A}_I.
\end{equation} Also, for any object $N\in D_{\perf}(A)$ we have an isomorphism
\begin{equation}\label{eq:completion_is_tensoring}N\Ltens{A}\hhat{A}_I\cong\holim\limits_n N\Ltens{A}A/I^n.\end{equation}
Combining \eqref{eq:tensoring_torsion_with_completion} and \eqref{eq:completion_is_tensoring} we conclude that for any $N\in D_{\perf,I\mhyphen tors}(A)$ we have an isomorphism $N\xto{\sim}\holim_n N\Ltens{A} A/I^n.$ This proves the lemma.\end{proof}

\begin{lemma}\label{lem:coh_Z_homotopy_limits_affine} Suppose that $X$ is affine, and $\cG\in D^b_{coh,Z}(X).$ Then the natural map
$\cG\to\holim\limits_{n}\iota_{n*}\bL\iota_n^*\cG$ is an isomorphism in $D(X).$\end{lemma}

\begin{proof} We keep the notation of the proof of Lemma \ref{lem:fully_faithful_for_perfect_affine}. Taking \eqref{eq:tensoring_torsion_with_completion} into account, it suffices to prove that for any $N\in D^b_{f.g.}(A)$ we have an isomorphism $N\Ltens{A}\hhat{A}_I\xto{\sim}\holim_n N\Ltens{A} A/I^n.$

Take any positive integer $l,$ and choose a perfect complex $N'\in\Perf(A)$ and a morphism $N'\to N$ which induces an isomorphism $H^{\geq -l}(N')\xto{\sim}H^{\geq -l}(N).$  By flatness of $\hhat{A}_I,$ we have an isomorphism $H^{\geq -l}(N'\Ltens{A}\hhat{A}_I)\xto{\sim}H^{\geq -l}(N)\Ltens{A}\hhat{A}_I.$ Further, since sequential inverse limit has cohomological dimension $1$ (that is $\bR\lim_n^{\geq 2}(B_n)=0$ for any inverse system $B_1\leftto B_2\leftto\dots$), we have an isomorphism $H^{\geq (-l+1)}(\holim_n N'\Ltens{A} A/I^n)\xto{\sim}H^{\geq (-l+1)}(\holim_n N\Ltens{A} A/I^n).$ Taking into account the isomorphism \eqref{eq:completion_is_tensoring} (for $N'$ instead of $N$), we conclude that we have an isomorphism $H^{\geq (-l+1)}(N\Ltens{A}\hhat{A}_I)\to H^{\geq (-l+1)}(\holim_n N\Ltens{A} A/I^n).$ Since $l$ can be chosen arbitrarily large, the morphism $N\Ltens{A}\hhat{A}_I\xto{\sim}\holim_n N\Ltens{A} A/I^n$ is an isomorphism in $D(\mk).$ This proves the lemma.\end{proof}

\begin{remark}\label{rem:compact_approximation} In the proof of Lemma \ref{lem:coh_Z_homotopy_limits_affine} we applied by now standard compact approximation technique. It has been used in various contexts by different authors, e.g. \cite{LN}, \cite{LO}. See also appendix \ref{app:boundedness_etc} of the present paper.\end{remark}

\begin{lemma}\label{lem:coh_Z_homotopy_limits}For a general $X,$ for any $\cG\in D^b_{coh,Z}(X),$ we have an isomorphism $\cG\cong\holim\limits_{n}\iota_{n*}\iota_n^*\cG$ in $D(X).$\end{lemma}

\begin{proof}We reduce the statement to the case of affine $X,$ which was established in Lemma \ref{lem:coh_Z_homotopy_limits_affine}. For that, suppose that we have an open cover $U_1\cup U_2=X,$ such that the statement is valid for pairs $(U_1,Z\cap U_1),$ $(U_2,Z\cap U_2)$ and $(U_{12},Z\cap U_{12}),$ where $U_{12}=U_1\cap U_2.$ It suffices to show that this implies the statement of the lemma for the pair $(X,Z).$ Indeed, the general case then follows from Lemma \ref{lem:coh_Z_homotopy_limits_affine} by induction on the number of subsets in an affine cover of $X.$

We denote the tautological open embeddings by $j_i:U_i\hto X$ and $j_{12}:U_{12}\hto X.$ We have the Mayer-Vietoris triangle in $D(X):$
\begin{equation}
\label{eq:Mayer-Vietoris_for_G}\cG\to \bR j_{1*}j_1^*\cG\oplus \bR j_{2*}j_2^*\cG\to \bR j_{12*}j_{12}^*\cG\to.
\end{equation}
Let us introduce some more notation for closed embeddings: $\iota_n^i:Z_n\cap U_i\hto U_i,$ and $\iota_n^{12}:Z_n\cap U_{12}\hto U_{12}.$ By the commutation of derived direct image with homotopy limits, and from the natural isomorphisms $\iota_{n*}\bL\iota_n^*\bR j_{i*}\cong \bR j_{i*}\iota_{n*}^i\bL\iota_n^{i*}$ (and similarly for $j_{12},$ $\iota_n^{12}$), we obtain the exact triangle
\begin{multline}
\label{eq:Mayer-Vietoris_for_holim}\holim\limits_{n}\iota_{n*}\bL\iota_n^*\cG\to \bR j_{1*}\holim\limits_{n}\iota_{n*}^1\bL\iota_n^{1*}j_1^*\cG\oplus \bR j_{2*}\holim\limits_{n}\iota_{n*}^2\bL\iota_n^{2*}j_2^*\cG\\
\to \bR j_{12*}\holim\limits_{n}\iota_{n*}^{12}\bL \iota_n^{12*}j_{12}^*\cG\to 
\end{multline}
The triangle \eqref{eq:Mayer-Vietoris_for_G} naturally maps to \eqref{eq:Mayer-Vietoris_for_holim}. By our assumption on the open cover, it induces an isomorphism on the middle and the right terms. Therefore, it induces an isomorphism on the left terms. This proves the lemma.
\end{proof}

Let $\cF\in D_Z(X)$ be an object. Then by adjunction we have a natural morphism
\begin{equation}
\label{eq:map_to_H_Z_of_holim}\cF\to\cH_Z(\holim\limits_{n}\iota_{n*}\bL\iota_n^*\cF).
\end{equation}

\begin{lemma}\label{lem:map_to_H_Z_of_holim_iso}For $\cF\in T_Z$ the morphism \eqref{eq:map_to_H_Z_of_holim} is an isomorphism.\end{lemma}

\begin{proof}It suffices to check that for any $\cG\in D_{\perf,Z}(X)$ the morphism
$$\cG\Ltens{\cO_X}\cF\to\cG\Ltens{\cO_X}\holim\limits_{n}\iota_{n*}\bL\iota_n^*\cF$$
is an isomorphism. Note that by perfectness of $\cG$ the RHS is isomorphic to $$\holim\limits_{n}(\cG\Ltens{\cO_X}\iota_{n*}\bL\iota_n^*\cF)\cong \holim\limits_{n}\iota_{n*}\bL\iota_n^*(\cG\Ltens{\cO_X}\cF).$$ The assertion now follows from Lemma \ref{lem:coh_Z_homotopy_limits} applied to $\cG\Ltens{\cO_X}\cF\in D_{\perf,Z}(X).$\end{proof}



Let $\cM\in D_{\perf}(\hhat{X}_Z)$ be an object, and denote by $\cM_n\in D_{\perf}(Z_n)$ its image. Thus, we have structural isomorphisms $\bL\iota_{k,n}^*\cM_n\cong\cM_k$ for $k<n.$

\begin{lemma}\label{lem:tensoring_by_holim}For any $k>0$ and $\cN\in D^b_{coh}(Z_k)$ we have $$\iota_{k*}\cN\Ltens{\cO_X}\holim\limits_{n}\bL\iota_{n*}\cM_n\cong \iota_{k*}(\cN\Ltens{\cO_{Z_k}}\cM_k).$$\end{lemma}

\begin{proof}Similarly to the proof of Lemma \ref{lem:coh_Z_homotopy_limits}, we may assume that $X$ is affine. 
By Proposition \ref{prop:completion_affine} the perfect complex $\cM\in\Perf(\hhat{X}_Z)$ is algebraizable, hence we may assume that $\cM=\cO_{\hhat{X}_Z}.$ In this case the assertion follows from the isomorphism \eqref{eq:tensoring_torsion_with_completion}.\end{proof}

\begin{lemma}\label{lem:restriction_of_holim}For any $k>0$ we have an isomorphism $\bL\iota_k^*\holim\limits_{n}\iota_{n*}\cM_n\cong\cM_k.$\end{lemma}

\begin{proof}Again, analogously to the proof of the previous lemma, we may assume that $X$ is affine and $\cM=\cO_{\hhat{X}_Z}.$ Let $A$ and $I$ be as in the proof of Lemma \ref{lem:fully_faithful_for_perfect_affine}. Then the assertion follows from the flatness of $\hhat{A}_I$ over $A.$\end{proof}

\begin{lemma}\label{lem:inverse_functor}Put $\cF:=\cH_Z(\holim\limits_{n}\iota_{n*}\cM_n).$ Then $\cF$ is in $T_Z.$\end{lemma}

\begin{proof}Let $\cG\in D_{\perf,Z}(X)$ be an object. We can find (and fix) $n_0>0$ and $\cG'\in D^b_{coh}(Z_{n_0})$ such that $\iota_{n_0*}\cG'\cong\cG.$ Applying Lemma \ref{lem:tensoring_by_holim} and projection formula, we obtain a chain of isomorphisms
$$
\cG\Ltens{\cO_X}\cF=\cG\Ltens{\cO_X}\cH_Z(\holim\limits_{n}\iota_{n*}\cM_n)\cong \iota_{n_0*}\cG'\Ltens{\cO_X}\holim\limits_{n}\iota_{n*}\cM_n\cong \iota_{n_0,*}(\cG'\Ltens{\cO_{Z_{n_0}}}\cM_{n_0}).
$$
By our assumptions, $\cG=\iota_{n_0,*}\cG'$ is a perfect complex on $X,$ and $\cM_{n_0}$ is a perfect complex on $Z_{n_0}.$ We conclude that  $\iota_{n_0,*}(\cG'\Ltens{\cO_{Z_{n_0}}}\cM_{n_0})$ is also perfect complex on $X.$ This proves the lemma.\end{proof}


\begin{proof}[Proof of Theorem \ref{th:perf_of_formal_is_T_Z}] From Lemma \ref{lem:restriction_is_perfect} we see in particular that $\iota_n^*(\cT_Z)\subset\Perf(Z_n).$ Therefore, we obtain a functor $$F:\cT_Z\to\holim\limits_n\Perf(Z_n)=\Perf(\hhat{X}_Z).$$
We need to prove that $F$ is a quasi-equivalence.

We first check that $F$ is quasi-fully-faithful.
Let $\cF_1,\cF_2\in\cT_Z$ be two objects. We have \begin{multline}\label{eq:chain_of_iso_holimits_adjunctions}\Hom_{\Perf(\hhat{X}_Z)}(F(\cF_1),F(\cF_2))=\holim\limits_{n}\Hom_{\Perf(Z_n)}(\iota_n^*\cF_1,\iota_n^*\cF_2)\simeq\\ \holim\limits\bR\Hom_X(\cF_1,\iota_{n*}\iota_n^*\cF_2)\cong \bR\Hom_X(\cF_1,\holim\limits_{n}\iota_{n*}\iota_n^*\cF_2)\cong\\ \bR\Hom_X(\cF_1,\cH_Z(\holim\limits_{n}\iota_{n*}\iota_n^*\cF_2))\cong \bR\Hom_X(\cF_1,\cF_2).\end{multline}
The last isomorphism of \eqref{eq:chain_of_iso_holimits_adjunctions} follows from Lemma \ref{lem:map_to_H_Z_of_holim_iso}. Therefore, the functor $F$ is indeed quasi-fully-faithful.

We now show that $F$ is essentially surjective. As above, let $\cM\in\Perf(\hhat{X}_Z)$ be an object, and denote by $\cM_k\in\Perf(Z_k),$ $k>0,$ its pullbacks. It follows from Lemma \ref{lem:inverse_functor}, we have an exact functor $G:D_{\perf}(\hat{X}_Z)\to T_Z,$ given by
$G(\cM):=\cH_Z(\holim\limits_{k>0}\iota_{k*}\cM_k)$ is in $\cT_Z.$ It is easy to see that $G$ is right adjoint to $[F].$ It suffices to show that the adjunction morphism $[F](G(\cM))\to\cM$ is an isomorphism.
Applying Lemma \ref{lem:restriction_of_holim}, we obtain
$$\bL\iota_k^* G(\cM)=\bL\iota_k^*\cH_Z(\holim\limits_{n}\iota_{n*}\cM_n)\cong \bL\iota_k^*\holim\limits_{n}\iota_{n*}\cM_n\cong\cM_n.$$
This proves that $F$ is a quasi-equivalence.

Commutativity of the triangle \eqref{eq:compatibility_restrictions_geometric} follows from the vanishing of each functor $\bL\iota_n^*$ on the subcategory $j_*D(X-Z)\subset D(X),$ where $j$ is the tautological open embedding. The theorem is proved.
\end{proof}

\subsection{Special case: $X$ is smooth and proper}
\label{ssec:X_smooth_proper}

We now specialize to the case when the ambient scheme $X$ is smooth and proper. If $T$ is a locally proper DG category, we put
$$\mD_{sg}(T):=\PsPerf(T)/\Perf(T),$$
and we denote by $\bbar{\mD_{sg}}(T)$ the Karoubi envelope of $\mD_{sg}(T).$ The corresponding triangulated categories are denoted by $D_{sg}(T)$ and $\bbar{D_{sg}}(T).$

\begin{remark}1) A justification of the above notation is the following geometric intuition. For $Y$ a proper scheme over $\mk$ and $T=\Perf(Y),$ by Proposition \ref{prop:pseudo_perfect_geometric} we have $\PsPerf(T)\simeq \mD^b_{coh}(Y),$ hence $\mD_{sg}(T)\simeq\mD_{sg}(Y)$ where $\mD_{sg}(Y)$ is (an enhancement of) the Orlov's triangulated category of singularities.

2) Of course, if $T$ is smooth and proper, then $D_{sg}(T)=0.$\end{remark} 

We start with application of Theorem \ref{th:perf_of_formal_is_T_Z} to the smooth and proper case. 

\begin{theo}\label{th:neighborhoods_and_D_sg} Let $X$ be smooth and proper scheme over $\mk,$ and $Z\subset X$ a closed subscheme. Let us put $\cS:=\Perf_Z(X).$ Then we have a natural equivalence $\Perf(\hhat{X}_Z)\simeq \PsPerf(\cS),$ compatible with the inclusions of $S.$ Passing to (Karoubi envelopes of) the quotients, we get a quasi-equivalence
$\Perf(\hhat{X}_Z-Z)\cong \bbar{\mD_{sg}}(\cS).$\end{theo}

\begin{proof}Recall that the triangulated category $D_Z(X)$ is compactly generated by $D_{\perf,Z}(X).$ Thus, we have a quasi-equivalence $\mD_Z(X)\simeq \Mod_{\cS}.$ We claim that this quasi-equivalence identifies $\cT_Z$ and $\PsPerf(\cS).$

Indeed, if $\cF\in T_Z,$ then for any object $\cG\in D_{\perf,Z}(X)$ we have 
$$\bR\Hom(\cG,\cF)\cong \bR\Gamma(X,\cG^{\vee}\Ltens{\cO_X}\cF)\in D_{\perf}(k),$$
since $\cG^{\vee}\Ltens{\cO_X} \cF\in D^b_{coh}(X)$ and $X$ is proper.

Conversely, let $\cF\in D_Z(X)$ be an object such that for any $\cG\in D_{\perf,Z}(X)$ we have $\bR\Hom(\cG,\cF)\in D_{\perf}(k).$ Take some perfect complex $\cG_1\in D_{\perf,Z}(X).$ Then for any perfect complex $\cG_2\in D_{\perf}(X)$ we have $$\bR\Hom(\cG_2,\cG_1\Ltens{\cO_X}\cF)\cong \bR\Hom(\cG_2\Ltens{\cO_X}\cG_1^{\vee},\cF)\in D_{\perf}(\mk).$$
It follows that $\cG_1\Ltens{\cO_X}\cF$ is in $D^b_{coh}(X),$ hence also in $D_{\perf,Z}(X).$ Therefore, $\cF\in T_Z.$

Applying Theorem \ref{th:perf_of_formal_is_T_Z}, we obtain a quasi-equivalence $\Perf(\hhat{X}_Z)\simeq \PsPerf(\cS),$ which by construction is compatible with inclusions of $\cS.$ This proves the theorem.\end{proof}

We need the following general statement on $\mD_{sg}.$

\begin{prop}\label{prop:D_sg_for_gluing_with_smooth_and_proper}Let $T$ be a locally proper DG category, and $\cC\subset T$ a full subcategory which is smooth and proper. Denote by $q:T\to T/\cC$ the quotient functor. Then 

1) the DG category $T/\cC$ is locally proper;

2) the DG functor $\bL q^*:D(T)\to D(T/\cC)$ takes $D_{\pspe}(T)$ to $D_{\pspe}(T/\cC);$

3) the quasi-functor $q_*:D(T/\cC)\to D(T)$ takes $D_{\perf}(T/\cC)$ to $D_{\perf}(T).$

4) the induced (quasi)-functors $\bbar{q^*}:\mD_{sg}(T)\to\mD_{sg}(T/\cC),$ $\bbar{q_*}:\mD_{sg}(T/\cC)\to \mD_{sg}(T)$ are quasi-inverse quasi-equivalences. 
\end{prop}

\begin{proof}It is well-known that under our assumptions the subcategory $D_{\perf}(\cC)\subset D_{\perf}(T)$ is admissible. Hence, the functor $\bL q^*:D_{\perf}(T)\to D_{\perf}(T/\cC)$ has both left and right adjoint functors, which are fully faithful. This proves 1) and 3).

2) follows from 3), since $\bL q^*(-)=-\Ltens{T}T/\cC.$

4) is deduced from 2) and 3) as follows. We have a semi-orthogonal decomposition $D_{\pspe}(T)=\langle q_* D_{\pspe}(T/\cC), D_{\perf}(\cC)\rangle,$ and the functor $\bL q^*:D_{\pspe}(T)\to D_{\pspe}(T/\cC)$ identifies with the left semi-orthogonal projection. The subcategory $D_{\perf}(T)\subset D_{\pspe}(T)$ is compatible with this decomposition: $D_{\perf}(T)=\langle q_* D_{\perf}(T), D_{\perf}(\cC)\rangle.$ The assertion follows.\end{proof}

\begin{cor}\label{cor:well_defined_V_infty}1) Let $f:X\to Y$ be a morphism between smooth and proper schemes, and $Z\subset Y$ a closed subset such that $f$ restricts to an isomorphism $(X-f^{-1}(Z))\xto{\sim}(Y-Z).$ Then we have a natural quasi-equivalence $\Perf(\hat{X}_{f^{-1}(Z)}-f^{-1}(Z))\simeq \Perf(\hat{Y}_Z-Z).$

2) If $\cha \mk=0,$ then for a smooth variety $V$ over $\mk$ the DG category $\Perf(\hhat{V}_{\infty})$ is well defined up to a quasi-equivalence.\end{cor}

\begin{proof}1) It is well known that under our assumptions we have $\bR f_*\cO_X\cong \cO_Y,$ hence the functor $\bL f^*$ is fully faithful. Let us put $\cC=\Ker(\bR f_*:\Perf(X)\to \Perf(Y)).$ Then we have a semi-orthogonal decomposition $D_{\perf}(X)=\langle [\cC], \bL f^*D_{\perf}(Y),$ in particular, $\cC$ is smooth and proper. Since $\cC\subset\Perf_{f^{-1}(Z)}(X),$ we also have an SOD $D_{\perf,f^{-1}(Z)}(X)=\langle [\cC], \bL f^* D_{\perf,Z}(X)\rangle,$ in particular, we have a quasi-equivalence $\Perf_{f^{-1}(Z)}(X)/\cC\simeq\Perf_Z(X).$ Now the assertion follows directly from Theorem \ref{th:neighborhoods_and_D_sg} and Proposition \ref{prop:D_sg_for_gluing_with_smooth_and_proper}.

2) Since $\cha \mk=0,$ by \cite{H, Nag} there exists a smooth compactification $\bbar{V}\supset V.$ Put $Z:=\bbar{V}-V.$ The DG category $\Perf(\hhat{V}_{\infty})$ is by definition $\Perf(\hhat{V}-Z).$ If $\bbar{V}'\supset V$ is another compactification, with $Z'=\bbar{V}'-V,$ and $f:\bbar{V}'\to \bbar{V}$ is a morphism which restricts to identity on $V,$ then by part 1) we have a quasi-equivalence $\Perf(\hhat{V}-Z)\simeq \Perf(\hhat{V}'-Z').$ Since for any two smooth compactifications we can find a "roof" of such morphisms, the assertion follows.\end{proof}

We conclude this section with a conjectural generalization of Corollary \ref{cor:well_defined_V_infty} 1) that seems to require a bit more advanced technique.

\begin{conj}\label{conj:general_proper_morphism_punctured neighborhoods} Let $f:X\to Y$ be a proper morphism of noetherian schemes, and $Z\subset Y$ a closed subset, such that $f$ restricts to an isomorphism $(X-f^{-1}(Z))\xto{\sim} (Y-Z).$ Then the pullback functor $\Perf(\hhat{Y}-Z)\to \Perf(\hhat{X}-f^{-1}(Z))$ is a quasi-equivalence.\end{conj}

\section{The main construction}
\label{sec:DG_cat_neighborhoods_main_result}

Let now $\cS\xto{\iota}\cA\xto{q}\cB$ be an abstract Morita short exact sequence (Definition \ref{defi:short_exact_dg}),
with $\cA$ smooth and proper. Note that automatically $\cS$ is locally proper and $\cB$ is smooth. Motivated by Theorem \ref{th:neighborhoods_and_D_sg}, we will call the DG category $\PsPerf(\cS)$ (resp. $\bbar{\mD_{sg}}(\cS)$) a formal neighborhood (resp. punctured neighborhood) of $\cS$ in $\cA.$ The restriction of scalars functor $\Mod_{\cA}\to \Mod_{\cS}$ induces a functor
\begin{equation}\label{eq:restriction_functor_formal}\cH_{\cS}:\Perf(\cA)\to \PsPerf(\cS).\end{equation}
Passing to the quotients by $\Perf(\cS),$ we obtain a functor
\begin{equation}\label{eq:restriction_punctured}\bbar{\cH_{\cS}}:\Perf(\cB)\to \bbar{\mD_{sg}}(\cS).\end{equation}

The main result of this section states roughly the following.

{\noindent{\bf Claim.}} {\it The category $\bbar{D_{sg}}(\cS)$ and the functor \eqref{eq:restriction_punctured} depend only on (the Morita equivalence class of) $\cB,$ and they can be constructed when $\cB$ is an arbitrary smooth DG category over $\mk.$}

We will now give such a description, and then formulate and proof the precise statement. First let us look at the DG category $\bbar{\mD_{sg}}(\cS)$ from a different point of view. 
Since $\cS$ is locally proper, we have a natural functor
\begin{equation}\label{eq:Calkin_to_Fun_to_Calkin}G_{\cS}:\bbar{\Calk}_{\cS}\to \bR\un{\Hom}(\cS^{op},\Calk_{\mk}).\end{equation} Indeed, it is induced by the projection
$$\Mod_{\cS}\simeq \bR\un{\Hom}(\cS^{op},\Mod_{\mk})\to \bR\un{\Hom}(\cS^{op},\Calk_{\mk}).$$
The latter functor vanishes on $\Perf(\cS)$ by the local properness of $\cS,$ hence the functor \eqref{eq:Calkin_to_Fun_to_Calkin} is well-defined. 
Tautologically, we have
\begin{equation}\label{eq:description_D_sg_Calkin}\bbar{\mD_{sg}}(\cS)\simeq\Ker(G_{\cS})).\end{equation}

Now we describe an analogous functor (but in the opposite direction) for an arbitrary smooth DG category $\cC.$ Namely, we define a functor $F_{\cC}:\bR\un{\Hom}(\cC^{op},\Calk_{\mk})\to \bbar{\Calk}_{\cC}$ to be the composition
\begin{equation}\label{eq:def_of_F_C}\bR\un{\Hom}(\cC^{op},\Calk_{\mk})\hto\Perf(\cC\otimes\Calk_k)\to \bbar{\Calk}_{\cC}.\end{equation}
Here the first arrow is the natural inclusion which is well-defined by the smoothness of $\cC.$ The second arrow is induced by the DG functor $\cC\otimes\Mod_{\mk}\to \Mod_{\cC},$ given by $X\otimes V\mapsto \bY(X)\otimes V.$ 

We will need the following observation.

\begin{prop}\label{prop:F_C_G_C_quasi-inverse}For a smooth and proper DG category $\cC,$ the functors $G_{\cC}$ and $F_{\cC}$ induce mutually inverse quasi-equivalences.\end{prop}

\begin{proof}Indeed, this follows straightforwardly from the identification $\Perf(\cC\otimes T)\simeq\bR\un{\Hom}(\cC^{op},T)$ for an arbitrary triangulated DG category $T.$\end{proof}

\begin{theo}\label{th:main_theorem}
For a short exact sequence as above, we have a quasi-equivalence
\begin{equation}\label{eq:main_equivalence}\bbar{\mD_{sg}}(\cS)\simeq\Ker(F_{\cB}).\end{equation} Moreover, the following diagram commutes: 
\begin{equation}\label{eq:compatibility_of_restrictions}\xymatrix{
\Perf(\cB)\ar[r]^{\bbar{\cH_{\cS}}}\ar@{^{(}->}[d]^{[-1]} & \bbar{\mD_{sg}}(\cS)\ar@{}[r]|*=0[@]{\simeq} & \Ker(F_{\cB})\ar@{^{(}->}[d]\\
\bR\un{\Hom}(\cB^{op},Mod_{\mk})\ar[rr] & & \bR\un{\Hom}(\cB,\Calk_{\mk}) 
}
\end{equation}
\end{theo}

\begin{proof}We have the following commutative diagram of (quasi-)functors:
\begin{equation}\label{eq:comm_diag_Calkins}
\xymatrix{\bbar{\Calk}_{\cS}\ar@{^{(}->}[r]^{\bbar{\Calk}_{\iota}}\ar[d]_{G_{\cS}} & \bbar{\Calk}_{\cA}\ar[r]^{\bbar{\Calk}_q}\ar@<-2pt>[d]_{G_{\cA}} & \bbar{\Calk}_{\cB}\\
\bR\un{\Hom}(\cS^{op},\Calk_{\mk}) & \bR\un{\Hom}(\cA^{op},\Calk_{\mk})\ar@<-2pt>[u]_{F_{\cA}}\ar[l]_{\iota_*} & \bR\un{\Hom}(\cB^{op},\Calk_{\mk})\ar@{_{(}->}[l]_{q_*}\ar[u]^{F_{\cB}}}
\end{equation}
By Proposition \ref{prop:short_exact_Calkins}, the functor $\bbar{\Calk}_{\iota}$ is quasi-fully-faithful, and it induce a quasi-equivalence $\bbar{\Calk}_{\cS}\simeq \Ker(\bbar{\Calk}_q).$ Also, the functor $q_*$ is  quasi-fully-faithful, and we have quasi-equivalences \begin{equation}\label{eq:one_more_embedding}\bR\un{\Hom}(\cB^{op},\Calk_{\mk})\simeq \Ker(\iota_*)\simeq\Ker(\iota_* G_{\cA})\subset \bbar{\Calk}_{\cA}.\end{equation} Taking \eqref{eq:description_D_sg_Calkin} into account, we may identify $\bbar{\mD_{sg}}(\cS)$ with the intersection of kernels $\Ker(\Calk_q)\cap \Ker(\iota_*G_{\cA})\subset\bbar{\Calk}_{\cA}.$ Now using \eqref{eq:one_more_embedding} and applying commutativity of \eqref{eq:comm_diag_Calkins}, we obtain
$$\bbar{\mD_{sg}}(\cS)\simeq\Ker(\Ker(\iota_*G_{\cA})\xto{\bbar{\Calk}_q}\bbar{\Calk}_{\cB})\simeq\Ker(\bR\un{\Hom}(\cB^{op},\Calk_{\mk})\xto{F_{\cB}}\bbar{\Calk}_{\cB}).$$
This proves the first assertion.

Let us describe the identification \eqref{eq:main_equivalence} more explicitly on objects. Take an object $M\in \bbar{\mD_{sg}}(S).$ It can be considered as an object of a bigger category $\bbar{\Calk}_{\cS}.$ Applying the functor $\bbar{\Calk}_{\iota},$ we obtain an object of $\bbar{\Calk}_{\cA}.$ Applying the functor $G_{\cA},$ we obtain an object of $\bR\un{\Hom}(\cA^{op},\Calk_{\mk}).$  It is contained in the image of the fully faithful functor $q_*:\bR\un{\Hom}(\cB^{op},\Calk_{\mk})\hto \bR\un{\Hom}(\cA^{op},\Calk_{\mk}).$ Hence, we obtain an object $N\in \bR\un{\Hom}(\cB^{op},\Calk_{\mk}),$ which is in fact contained in $\Ker(F_{\cB}).$

To prove the second assertion, let us note that we have an exact triangle in $D(\cA\otimes\cA^{op}):$ 
$$\cA\Ltens{\cS}{\cA}\to \cA\to \cB.$$ This implies commutativity of the following diagram
of functors:
\begin{equation}\label{eq:comm_diag_H_S_and_so_on}\xymatrix{\Perf(\cB)\ar[r]^{\bbar{\cH_{\cS}}}\ar@{^{(}->}[d] & \bbar{\mD_{sg}}(\cS)\ar@{^{(}->}[r] & \bbar{\Calk}_{\cS}\ar[d]^{\bbar{\Calk}_{\iota}}\\
\Mod_{\cB}\ar[r]^{q_*[-1]} & \Mod_{\cA}\ar[r] & \bbar{\Calk}_{\cA}.}\end{equation}
Also, the following diagram commutes:
\begin{equation}\label{eq:comm_diag_various_restrictions}
\xymatrix{\bR\un{Hom}(\cB^{op},\Mod_{\mk})\ar[d]\ar@{}[r]|*[@]{\simeq} & \Mod_{\cB}\ar[r]^{q_*} & \Mod_{\cA}\ar[d]\\
\bR\un{Hom}(\cB^{op},\Calk_{\mk})\ar@{^{(}->}[r]^{q_*} & \bR\un{Hom}(\cA^{op},\Calk_{\mk})\ar[r]^-{F_{\cA}} & \bbar{\Calk}_{\cA}
}
\end{equation}
Combining \eqref{eq:comm_diag_H_S_and_so_on}, \eqref{eq:comm_diag_various_restrictions}, and running through the construction of the identification \eqref{eq:main_equivalence}, we obtain the commutativity of \eqref{eq:compatibility_of_restrictions}.
\end{proof}

Theorems \ref{th:neighborhoods_and_D_sg} and \ref{th:main_theorem} motivate the following definitions.
We denote by $\bbar{\bY}:\cB\to\Ker(F_B)$ the functor given by the composition $\cB\xto{\bY}\bR\un{\Hom}(\cB^{op},\Mod_{\mk})\to \bR\un{\Hom}(\cB^{op},\Calk_{\mk}),$ and similarly for $\bbar{\bY}^*:\Perf(\cB)\to\Ker(F_{\cB}).$

\begin{defi}\label{def:neighborhood_of_infinity}Let $\cB$ be a smooth DG category. 

1) The (triangulated) DG category $\Perf_{top}(\hhat{\cB}_{\infty}):=\Ker(F_{\cB}:\bR\un{\Hom}(\cB^{op},\Calk_{\mk})\to \Calk_{\cB})$ is called the category of perfect complexes on the formal punctured neighborhood of infinity of $\cB.$

2) We define the DG subcategory $\Perf_{alg}(\hhat{\cB}_{\infty})\subset \Perf(\hhat{\cB}_{\infty})$ to be generated as a triangulated DG subcategory by the image of $\bbar{\bY}:\cB\to\Perf_{top}(\hhat{\cB}_{\infty}).$ It is called the category of algebraizable perfect complexes on the formal punctured neighborhood of infinity of $\cB.$

3) Finally, we define the DG category $\hhat{\cB}_{\infty}$ to be the essential image of $\bbar{\bY}.$\end{defi}

\begin{remark}Note that the DG category $\hhat{\cB}_{\infty}$ is not necessarily pre-triangulated. It follows directly from Definition \ref{def:neighborhood_of_infinity} that $\Perf_{alg}(\hhat{\cB}_{\infty})$ is quasi-equivalent to $\Perf(\hhat{\cB}_{\infty}).$ However, the category $\Perf_{top}(\hat{\cB}_{\infty})$ is in general strictly larger than $\Perf_{alg}(\hhat{\cB}_{\infty}).$\end{remark}

We first state that our categorical construction is compatible with algebro-geometric formal punctured neighborhood of infinity.

\begin{theo}\label{th:compatibility} Let $X$ be a smooth variety over $\mk$ such that there exists a smooth compactification $\bbar{X}\supset X$ (this holds automatically if $\cha \mk=0$). Putting $Z:=\bbar{X}-X$ and $\cB:=\Perf(X),$ we have $\Perf_{top}(\hhat{\cB}_{\infty})\simeq \Perf(\hhat{\bbar{X}}_Z-Z).$ In particular, if $\cha \mk=0,$ we have $\Perf_{top}(\hhat{\cB}_{\infty})\simeq \Perf(\hhat{X}_{\infty}).$\end{theo}

\begin{proof}This is formally deduced from the above results. Putting $\cS:=\Perf_Z(\bbar{X})$ and $\cA:=\Perf(\bbar{X}),$ we obtain a short exact sequence $\cS\to\cA\to\cB$ as above. By Theorem \ref{th:neighborhoods_and_D_sg}, we have $\Perf(\hhat{\bbar{X}}_Z-Z)\simeq \bbar{\mD_{sg}}(\cS).$ By Theorem \ref{th:main_theorem}, we have $\bbar{\mD_{sg}}(\cS)\simeq \Perf_{top}(\hhat{\cB}_{\infty}).$ This proves the theorem.\end{proof}


Note that for an arbitrary small DG category $\cB$ we can define $\hhat{\cB}_{\infty}$ to be the essential image of $\cB$ in $\bR\un{\Hom}(\cB^{op},\Calk_{\mk}).$ In particular, we always have a natural functor $\cB\to\hhat{\cB}_{\infty}.$ We now give a more direct description of $\hhat{\cB}_{\infty}.$

\begin{prop}\label{prop:neighborhood_of_infty_Hochschild_cochains}1) Let $B$ be a DG algebra. Then the DG algebra $\hhat{B}_{\infty}$ is quasi-isomorphic to the DG algebra $\hhat{B}_{\infty}^{can}=C^{\bullet}(B^{op},\Calk_{\mk}(B,B)),$ associated with the composition morphism of DG algebras $\cB^{op}\to\Hom_{\mk}(B,B)\to\Calk_{\mk}(B,B)$ (right action of $B$ on itself). The morphism $B\to C^{\bullet}(B^{op},\Calk_{\mk}(B,B))$ is associated with the composition morphism of DG algebras $B\otimes B^{op}\to\Hom_{\mk}(B,B)\to\Calk_{\mk}(B,B)$ ($B\mhyphen B$-bimodule structure on $B$).

2) More generally, if $\cB$ is a DG category, then the DG category $\hhat{\cB}_{\infty}$ has the following model $\hhat{\cB}_{\infty}^{can}$. The objects of $\hhat{\cB}_{\infty}^{can}$ are the same as the objects of $\cB.$ The morphisms are given by $\hhat{\cB}_{\infty}^{can}(x,y)=C^{\bullet}(\cB^{op},\Calk_{\mk}(\bY(x),\bY(y))).$ The composition in $\hhat{\cB}_{\infty}^{can}$ and the DG functor $\cB\to\hhat{\cB}_{\infty}^{can}$ are coming from the DG functor $\cB\otimes\cB^{op}\xto{I_{\cB}}\Mod_{\mk}\to\Calk_{\mk}.$
\end{prop}

\begin{proof}This essentially follows directly from the definition of $\hhat{\cB}_{\infty}.$ We will comment on part 2), and part 1) is a special case.

Let $T$ be any small DG category, and consider $\cB^{op}\otimes T$-modules $M,M'$ such that for all $X\in \cB$ the $T$-modules $M(X,-),$ $M'(X,-)$ are cofibrant. Then the Hochschild cochain complex $C^{\bullet}(\cB^{op},\Hom_T(M,M'))$ is the complex of morphisms in the DG category of $A_{\infty}$-functors from $\cB^{op}$ to $\Mod_{T}.$ In particular, this complex is quasi-isomorphic to $\bR\Hom_{\cB^{op}\otimes T}(M,M').$ It remains to apply this observation to the case when $T=\Calk_{\mk},$ and the bimodules $M,M'$ are associated with DG functors $\bY(x),\bY(y):\cB^{op}\to\Calk_{\mk}.$\end{proof}

For an arbitrary DG category $\cC$ we will call the objects of $[\bR\un{Hom}(\cC^{op},\Calk_{\mk})]$ the {\it almost DG $\cC$-modules}. If in addition $\cC$ is smooth, then we will call the objects of $[\Perf_{top}(\hhat{\cC}_{\infty})]$ (resp. $\Perf_{alg}(\hhat{\cC}_{\infty})$) the {\it (algebraizable) perfect almost DG $\cB$-modules}. In order to justify this terminology, let us consider for simplicity the case of a DG algebra $B$ and an object $M\in \Fun(B^{op},\Calk_{\mk})$ (the actual DG functor). Then $M$ is a complex of vector spaces equipped with a homomorphism of DG algebras $f:B^{op}\to\End_{\Calk_{\mk}}(M).$ Let us choose a lift of $f$ to a morphism  of graded algebras $\tilde{f}:(B^{op})^{\gr}\to\End_{\mk}(M^{gr}),$ and put $mb=m\cdot b:=(-1)^{|m|\cdot|b|}\tilde{f}(b)(m)$ for homogeneous elements $b\in B,$ $m\in M.$ Then the following conditions are necessary and sufficient for $f$ to be a homomorphism of DG algebras:
\begin{itemize}
\item $\rk(m\mapsto m\cdot 1_B-m)<\infty;$ 

\item $\rk(m\mapsto d(mb)-d(m)b-(-1)^{|m|}md(b))<\infty$ for a homogeneous $b\in B;$

\item $\rk(m\mapsto m(b_1b_2)-(mb_1)b_2)<\infty$ for homogeneous $b_1,b_2\in B.$
\end{itemize}
If all these ranks are equal to zero, then $\tilde{f}$ defines a DG $B$-module structure on $M.$ This justifies the terminology "almost DG module".


\section{General properties of $\hhat{\cB}_{\infty}$}
\label{sec:generalities_on_the_construction}

Let $\cB$ be a smooth DG category over $\mk.$ We start with the following observation.

\begin{prop}\label{prop:pseudo-perfect_in_the_kernel}We have $\Ker(\Perf(\cB)\xto{\bbar{\bY}^*}\Perf_{top}(\hhat{\cB}_{\infty}))=\PsPerf(\cB).$ In particular, we have a natural functor $\cB/\PsPerf(\cB)\to \hhat{\cB}_{\infty}.$\end{prop}

\begin{proof}We have \begin{multline*}\Ker(\bbar{\bY}^*)=\Perf(\cB)\cap\Ker(\Mod_{\cB}\to\bR\un{\Hom}(\cB^{op},\Calk_{\mk}))\\=\Perf(\cB)\cap\bR\un{\Hom}(\cB,\Perf(\mk))=\PsPerf(\cB).\end{multline*}
This proves the proposition.\end{proof}

Although the assignment $\cB\mapsto \Perf(\hhat{\cB}_{\infty})$ is not functorial in $\cB,$ we have the following partial functoriality result.

\begin{prop}\label{prop:partial_functoriality}Let $\Phi:\cB_1\to\cB_2$ be a functor between smooth DG categories, such that the functor $\bL\Phi^*:D_{perf}(\cB_1)\to D_{perf}(\cB_2)$ has a left adjoint. Then we have a natural functor $\hhat{\Phi}_{\infty}^*:\Perf_{top}(\hhat{\cB_1}_{\infty})\to\Perf_{top}(\hhat{\cB_2}_{\infty}),$ and the following square commutes:
\begin{equation}\label{eq:pullback_commutes_with_restrictions}
\xymatrix{\Perf(\cB_1)\ar[r]^{\Phi^*}\ar[d]^{\bbar{\bY}^*} & \Perf(\cB_2)\ar[d]^{\bbar{\bY}^*}\\
\Perf(\hhat{\cB_1}_{\infty})\ar[r]^{\hhat{\Phi}_{\infty}^*} & \Perf(\hhat{\cB_2}_{\infty}).
}
\end{equation}\end{prop}

\begin{proof}By our assumption on $\Phi,$ for any DG category $T$ the extension of scalars functor $(\Phi\otimes\id_T)^*:\Perf(\cB_1\otimes T)\to\Perf(\cB_2\otimes T)$ takes $\bR\un{\Hom}(\cB_1^{op},T)$ to $\bR\un{\Hom}(\cB_2^{op},T).$ Furthermore, the following square commutes
$$\xymatrix{\bR\un{\Hom}(\cB_1^{op},\Calk_{\mk})\ar[r]^{(\Phi\otimes\id_{\Calk_{\mk}})^*}\ar[d]^{F_{\cB_1}} & \bR\un{\Hom}(\cB_2^{op},\Calk_{\mk})\ar[d]^{F_{\cB_2}}\\
\bbar{\Calk}_{\cB_1}\ar[r]^{\bbar{\Calk}_{\Phi}} & \bbar{\Calk}_{\cB_2}.}$$ Hence, the upper horizontal functor induces a well-defined functor $\hhat{\Phi}_{\infty}^*:\Perf_{top}(\hhat{\cB_1}_{\infty})\to\Perf_{top}(\hhat{\cB_2}_{\infty}).$ The commutativity of \eqref{eq:pullback_commutes_with_restrictions} follows by constructions.\end{proof}

It is also technically useful to rewrite the $\cB\mhyphen\cB$-bimodule $C^{\bullet}(\cB^{op},\cB_{\cB}\otimes {}_{\cB}\cB^*)\simeq Fiber(\cB\to \hhat{\cB}_{\infty}).$

\begin{prop}\label{prop:rewriting_Hochschild_cochains}The natural composition morphism $$\cB^!\Ltens{\cB}\cB^*\xto{\sim} C^{\bullet}(\cB^{op},\cB_{\cB}\otimes {}_{\cB}\cB)\Ltens{\cB}\cB^*\to C^{\bullet}(\cB^{op},\cB_{\cB}\otimes {}_{\cB}\cB^*)$$ is an isomorphism in $D(\cB\otimes\cB^{op}).$\end{prop}

\begin{proof}This follows immediately from the compactness of $\cB$ in $D(\cB\otimes \cB^{op}).$\end{proof}

\section{The case of an associative algebra}
\label{sec:assoc_algebras}

Let $A$ be an associative $k$-algebra. In this section we take a closer look at the DG algebra $\hhat{A}_{\infty}.$ For an element $a\in A$ we denote by $L_a:A\to A$ (resp. $R_a:A\to A$) the $\mk$-linear operator given by $L_a(a')=aa'$ (resp. $R_a(a')=a'a$).

\begin{prop}\label{prop:descr_of_H^0_neighborhood}The DG algebra $\hhat{A}_{\infty}$ has non-negative cohomology, and its zero-th cohomology algebra has the following description:
\begin{equation}\label{eq:H^0_of_the_neighborhood}H^0(\hhat{A}_{\infty})\cong \{\varphi\in\Hom_k(A,A)\mid \forall a\in A \,\,\rk[\varphi,R_a]<\infty\}/A\otimes A^*.
\end{equation}
The homomorphism $H^0(\bar{Y}):A\to H^0(\hhat{A}_{\infty})$ is given by $a\mapsto \bbar{L_a}.$\end{prop}

\begin{proof}By Proposition \ref{prop:neighborhood_of_infty_Hochschild_cochains} 1), the graded cohomology algebra $H^{\bullet}(\hhat{A}_{\infty})$ is isomorphic to the graded algebra $HH^{\bullet}(A^{op},\Calk_{\mk}(A,A)).$ This immediately implies the vanishing of $H^{<0}(\hhat{A}_{\infty}).$ Furthermore, we have the injective homomorphism  $HH^0(A,\Calk_{\mk}(A,A))\hto\Calk_{\mk}(A,A)$ whose image consists of elements $\bbar{\varphi}\in\Calk_k(A,A)$ which commute with projections of $R_a\in\End_{\mk}(A)$ onto $\Calk_{\mk}(A,A),$ for all $a\in A.$ Clearly, the image of this homomorphism is identified with the RHS of \eqref{eq:H^0_of_the_neighborhood}. This proves the first assertion. The second assertion follows directly from Proposition \ref{prop:neighborhood_of_infty_Hochschild_cochains} 1).\end{proof}

The following lemma is somewhat trivial but it seems useful to formulate it for clarity.

\begin{lemma}\label{lem:suffices_on_generators} Let $S\subset A$ be any subset that generates $A$ as a $\mk$-algebra, and $\varphi:A\to A$ a $\mk$-linear operator. If $\rk [\varphi,R_s]<\infty$ for all $s\in S,$ then $\rk[\varphi,R_a]<\infty$ for all $a\in A.$\end{lemma}

\begin{proof}Indeed, let us note that $[\varphi,R_{ab}]=[\varphi,R_b]R_a+R_b[\varphi,R_a].$ Thus, if both $[\varphi,R_a]$ and $[\varphi,R_b]$ have finite rank, then so does $[\varphi,R_{ab}].$ The lemma follows.\end{proof}

We recall the notion of formal smoothness (aka quasi-freeness) for an associative algebra, due to Cuntz and Quillen \cite{CQ}.

\begin{defi}\label{def:Quillen_smooth}\cite[Definition 3.3 and Proposition 6.1]{CQ} A $\mk$-algebra $A$ is formally smooth if the following equivalent conditions are satisfied:

$\rm{(i)}$ For any $\mk$-algebra $B$ and a nilpotent ideal $I\subset B,$ any homomorphism $f:A\to B/I$ can be lifted to a homomorphism $\tilde{f}:A\to B.$

$\rm{(ii)}$ The projective dimension of the diagonal bimodule $A\in A\text{-Mod-}A$ is at most $1.$

$\rm{(ii)}$ The bimodule of differentials $\Omega_A=\Ker(A\otimes A\xto{m}A)$ is projective.\end{defi}

\begin{prop}\label{prop:Quillen_smooth}Suppose that $A$ is formally smooth over $\mk.$ Then we have $H^{\ne 0}(\hhat{A}_{\infty})=0.$ Therefore, $A$ is quasi-isomorphic to the associative algebra $H^0(\hhat{A}_{\infty}),$ which was described in Proposition \ref{prop:descr_of_H^0_neighborhood}.\end{prop}

\begin{proof}By Proposition \ref{prop:neighborhood_of_infty_Hochschild_cochains}, we have an exact triangle $$C^{\bullet}(A^{op},A\otimes A^*)\to A\to \hhat{A}_{\infty}$$ in $D(\mk).$ Since $A$ is formally smooth, we have $HH^{>1}(A^{op},A\otimes A^*)=0.$ From the long exact sequence in cohomology we see that $H^{\ne 0}(\hhat{A}_{\infty})=0.$\end{proof}

It is a pleasant exercise to compute the algebra $\hhat{A}_{\infty}$ in the following example.

\begin{example}\label{ex:computation_for_k[t]}Let us consider the case $A=k[t].$ Clearly, $A$ is formally smooth. By Theorems \ref{th:neighborhoods_and_D_sg} and \ref{th:main_theorem} we know that $\hhat{\mk[t]}_{\infty}\simeq \mk((t^{-1})).$ Let us construct an explicit identification of $\mk((t^{-1}))$ and the RHS of \eqref{eq:H^0_of_the_neighborhood} (for $A=k[t]$). For each $n\in \Z$ we define a linear operator $T_{n}:\mk[t]\to\mk[t]$ given by $$T_{n}(t^m)=\begin{cases}t^{m+n} & \text{for } m\geq \max(-n,0);\\
0 & \text{for }0\leq m\leq n-1.\end{cases}$$ 
Clearly, for $n\geq 0$ we have $T_n=L_{t^n}.$ We define a linear map $\phi:\mk((t^{-1}))\to\End_{\mk}(A)$ by the formula $\phi(\sum\limits_{-\infty}^{n=k}c_n t^n)(g)=\sum\limits_{-\infty}^{n=k}c_n T_n(g).$ The infinite sum in the RHS is well-defined since $T_n(g)=0$ for $n<-\deg(g).$ Moreover, for any $f\in\mk((t^{-1}))$ we have 
$\im([\phi(f),R_t])\subset \mk\subset \mk[t],$ hence $\rk[\phi(f),R_t]\leq 1<\infty.$ By Lemma \ref{lem:suffices_on_generators}, we have $\phi(f)\in\{\varphi\in\End_{\mk}(k[t])\mid\forall a\in \mk[t]\,\,\rk[\varphi,R_a]<\infty\}.$ Passing to the quotient by $\mk[t]\otimes \mk[t]^*$ we obtain a linear map from $k((t^{-1}))$ to the RHS of \eqref{eq:H^0_of_the_neighborhood}, which is easily checked to be an isomorphism of algebras.\end{example}

\begin{example}Similarly to Example \ref{ex:computation_for_k[t]} one can explicitly identify $\hhat{\mk[t^{\pm 1}]}_{\infty}$ with $\mk((t))\times\mk((t^{-1})).$\end{example}

\section{Example: affine curve}
\label{sec:affine_curve}

Let $C=\Spec B$ be a smooth connected affine curve over $\mk,$ and let $\bbar{C}\supset C$ be its (unique) smooth compactification. By Theorem \ref{th:compatibility} we have an isomorphism of algebras
$$\hhat{B}_{\infty}\cong\prod\limits_{p\in\bbar{C}-C}\hhat{K}_{\bbar{C},p},$$
where $\hhat{K}_{\bbar{C},p}$ denotes the complete local field at $p.$ We would like to describe the corresponding picture with almost DG $B$-modules.

\begin{prop}\label{prop:affine_curve_almost_modules}Let $X$ be a complete integral curve over $\mk$ (not necessarily smooth), and let $Y\subsetneq X$ be a nonempty open subscheme (hence affine). Then for any point $p\in X-Y$ the quotient space $\mk(X)/\cO_{X,p}\cong \hhat{K}_{X,p}/\hhat{\cO}_{X,p}$ is naturally an almost $\hhat{K}_{X,p}$-module, and by restriction of scalars also an almost $\mk(X)$-module and almost $\cO(Y)$-module. The natural map
\begin{equation}\label{eq:almost isomorphism_affine_curve}\cO(Y)\xto{u}\bigoplus\limits_{p\in X-Y}\mk(X)/\cO_{X,p}\end{equation} gives an isomorphism of almost $\cO(Y)$-modules.\end{prop}

\begin{proof}First we define an almost action of 
$\hhat{K}_{X,p}$ on $\hhat{K}_{X,p}/\hhat{\cO}_{X,p}.$ 
Let $f\in\hhat{K}_{X,p}$ be an element. We have 
an ideal $I_{f,p}\subset \hhat{\cO}_{X,p}$ consisting of elements $g$ such that $fg\in\hhat{\cO}_{X,p}.$ Clearly, $\dim(\hhat{\cO}_{x,p}/I_{f,p})<\infty,$ hence the morphism $\hhat{K}_{X,p}/I_{f,p}\xto{f}\hhat{K}_{X,p}/\cO_{X,p}$ (multiplication by $f$) gives a well defined element of $\End_{\Calk_{\mk}}(\mk(X)/\cO_{x,p}).$ The resulting map $\rho_p:\mk(X)\to \End_{\Calk_{\mk}}(\mk(X)/\cO_{X,p})$ is a homomorphism of algebras.

Now we prove the second assertion. By definition of almost $\cO(Y)$-actions, the map \eqref{eq:almost isomorphism_affine_curve} is almost $\cO(Y)$-linear. It remains to show that it has finite-dimensional kernel and cokernel. In fact we have $\Ker(u)\cong H^0(X,\cO_X),$ $\coker(u)\cong H^1(X,\cO_X),$ which follows from the acyclic resolution
$$0\to\cO_X\to j_*\cO_Y\to \bigoplus\limits_{p\in X-Y}i_{p*}\mk(X)/\cO_{x,p}\to 0.$$ Here $j:Y\to X$ and $i_p:\Spec(\cO_{X,p})\to X$ are the natural morphisms. Since $X$ is complete, we have $\dim H^{\bullet}(X,\cO_X)<\infty.$ This proves the proposition.\end{proof}

Returning to our situation, the almost action of $\prod\limits_{p\in\bbar{C}-C}\hhat{K}_{p}$ on $\cO(C)$ is given by the almost isomorphism \eqref{eq:almost isomorphism_affine_curve} (for $Y=C,$ $X=\bbar{C}$), and the almost actions of $\hhat{K}_{\bbar{C},p}$ on $\mk(\bbar{C})/\cO_{\bbar{C},p}.$ 

The following describes our construction for the algebra of rational functions on a curve.

\begin{prop}\label{prop:rational functions and adeles} Let $X$ be a smooth connected complete curve. Then we have a natural isomorphism $\hhat{\mk(X)}_{\infty}\cong \A_X,$ where $\A_X=\prod'\limits_{x\in X^{cl}}\hhat{K}_{X,x}$ is the algebra of adeles on $X.$\end{prop}

\begin{proof}We have a smooth compactification, given by a short exact sequence of DG categories $\Perf_{\tors}(X)\to \Perf(X)\to \Perf(\mk(X)),$ where $\Perf_{\tors}(X)$ denotes the full subcategory of perfect complexes with torsion cohomology. We also denote by $D_{tors}(X)\subset D(X)$ a similar subcategory in the large derived category. We have a functor $\cH_{\tors}:D(X)\to D_{\tors}(X),$ which is right adjoint to the inclusion. It is easy to see that $$\cH_{\tors}(\cO_X)\cong \mk(X)/\cO_X[-1]\cong\bigoplus\limits_{x\in X^{cl}}i_{x*}(\mk(X)/\cO_{X,p})[-1].$$ Therefore, $\bR\End(\cH_{\tors}(\cO_X))\cong\A_X^0:=\prod\limits_{x\in X^{cl}}\hhat{\cO}_{X,x}.$ To obtain $\hhat{\mk(X)}_{\infty},$ we need to take endomorphisms of (the projection of) $\A_X^0$ in the quotient of $\Perf(\A_X^0)$ by all the modules $\mk(x),$ $x\in X^{cl}.$

For a finite subset $S\in X^{cl}$ we have $$\bR\End_{\Perf(\A_X^0)/\langle\mk(x)\rangle_{x\in S}}(\A_X^0)\cong \A_{X,S}=\prod\limits_{x\in S}\hhat{K}_{X,x}\times \prod\limits_{x\in X^{cl}-S}\hhat{\cO}_{X,x}.$$ We conclude that $\hhat{\mk(X)}_{\infty}\cong \colim_{S}\A_{X,S}=\A_X.$\end{proof}

Again, we would like to describe explicitly the action of $\A_X$ on $\mk(X)$ as an almost $\mk(X)$-module. Arguing as in Proposition \ref{prop:affine_curve_almost_modules}, we obtain an almost $\mk(X)$-module structure on the direct sum $\bigoplus\limits_{x\in X^{cl}}\mk(X)/\cO_{X,p},$ and an isomorphism of almost $\mk(X)$-modules
\begin{equation}\label{eq:rational functions_and_principal_parts}\mk(X)\xto{\sim} \bigoplus\limits_{x\in X^{cl}}\hhat{K}_{X,x}/\hhat{\cO}_{X,x}.\end{equation} The almost action of $\A_X$ on the RHS of \eqref{eq:rational functions_and_principal_parts} is described as follows. Let $f=\{f_x\}_{x\in X^{cl}}\in\A_X$ be an element. Let us put $S_f:=\{x\in X^{cl}\mid f_x\not\in\hhat{\cO}_{X,x}\}$ (so that $S_f$ is finite). Componentwise multiplications by $f_x$ give a well-defined linear map $$\bigoplus\limits_{x\in S}\hhat{K}_{X,x}/f_x^{-1}\hhat{\cO}_{X,x}\oplus\bigoplus\limits_{x\in X^{cl}-S}\hhat{K}_{X,x}/\hhat{\cO}_{X,x}\to \bigoplus\limits_{x\in X^{cl}}\hhat{K}_{X,x}/\hhat{\cO}_{X,x}.$$ Since $\dim (\hhat{\cO}_{X,x}/f_x^{-1}\hhat{\cO}_{X,x})<\infty$ for $x\in S,$ we obtain a well-defined element of $\End_{\Calk_{\mk}}(\bigoplus\limits_{x\in X^{cl}}\hhat{K}_{X,x}/\hhat{\cO}_{X,x}).$ Moreover, this element commutes with the almost action of $\mk(X)$ on $\bigoplus\limits_{x\in X^{cl}}\hhat{K}_{X,x}/\hhat{\cO}_{X,x}.$ The resulting map $$\A_X\to\End_{\bR\un{\Hom}(\mk(X)^{op},Calk_{\mk})}(\bigoplus\limits_{x\in X^{cl}}\hhat{K}_{X,x}/\hhat{\cO}_{X,x})\cong \End_{\bR\un{\Hom}(\mk(X)^{op},Calk_{\mk})}(\mk(X))$$ is an isomorphism of algebras.  

\begin{prop}\label{prop:topological_perf_over_adeles}In the above notation, the triangulated category $[\Perf_{top}(\hhat{\mk(X)}_{\infty})]$ is equivalent to the Verdier quotient of $C=\prod\limits_{x\in X^{cl}}D_{\perf}(\hhat{\cO}_{X,x})$ by the subcategory formed by $\mk(x)\in D_{\perf}(\hhat{\cO}_{X,x})\subset C.$ In particular, $\Perf_{top}(\hhat{\mk(X)}_{\infty})\ne \Perf_{alg}(\hhat{\mk(X)}_{\infty})\simeq \Perf(\A_X).$\end{prop}

\begin{proof}Keeping notation of the proof of Proposition \ref{prop:rational functions and adeles}, let us put $\cS:=\Perf_{\tors}(X).$ By Theorem \ref{th:main_theorem} we have $\Perf_{top}(\hhat{\mk(X)}_{\infty})\simeq \bbar{\mD_{sg}}(\cS).$ Since $K_{-1}(\cS)=K_{-1}(\Coh_{\tors}(X))=0,$ we have $\bbar{\mD_{sg}}(\cS)\simeq \mD_{sg}(\cS).$ We have a quasi-equivalence $\bigoplus\limits_{x\in X^{cl}}\Perf_{\{x\}}(X)\simeq \Perf_{\tors}(X).$ It follows that $D_{\pspe}(\cS)\simeq \prod\limits_{x\in X^{cl}} D_{\pspe}(\Perf_{\{x\}}(X))\simeq \prod\limits_{x\in X^{cl}}D_{\perf}(\hhat{\cO}_{X,x})=C.$ Under this identification, the skyscraper sheaf $\cO_x$ corresponds to $\mk(x)\in C.$ This proves the proposition.\end{proof}

\begin{remark}\label{rem:more_on_perf_top_adeles} Proposition \ref{prop:topological_perf_over_adeles} can be illustrated explicitly by the following example. Let $s:X^{cl}\to\Z$ be any unbounded set-theoretic function.  The direct sum $\bigoplus\limits_{x\in X^{cl}}\mk(X)/\cO_{X,x}[s(x)]$ is naturally a perfect almost DG $\mk(X)$-module (with zero differential), but as an object of $[\Calk_{\mk}]$ it is unbounded, i.e. it is not contained in $D^b(\Vect_{\mk})/D_{\perf}(\mk).$ In particular, this almost DG module is not in $\Perf_{alg}(\hhat{\mk(X)}_{\infty}).$ 

This example shows that for a smooth cohomologically bounded DG algebra $B$ there might exist unbounded perfect almost DG $B$-modules.\end{remark}

\section{Example: smooth affine variety}
\label{sec:smooth_affine_variety}

Let $X=\Spec B$ be a smooth connected affine variety over $\mk,$ $\dim X\geq 2.$ We would like to describe the cohomology of the DG algebra $\hhat{B}_{\infty}.$

\begin{prop}\label{prop:smooth_affine_var_neighborhood_infty}We have an isomorphism of graded algebras $H^{\bullet}(\hhat{B}_{\infty})\cong B\oplus (\Omega^d_B)^*[1-d]$ (with the standard $B-B$-bimodule structure on $(\Omega^d_B)^*$).\end{prop}

\begin{proof}It suffices to prove that $C^{\bullet}(B^{op},B_B\otimes {}_B B^*)\cong (\Omega^d_B)^*[-d]$ as $B\mhyphen B$-bimodules. 
Using Proposition \ref{prop:rewriting_Hochschild_cochains}, we obtain a chain of isomorphisms in $D(B\otimes B^{op}):$ $$C^{\bullet}(B^{op},B_B\otimes {}_B B^*)\cong B^!\Ltens{B} B^*\cong\Lambda^d T_B[-d]\Ltens{B} B^*\cong (\Omega^d_B)^*[-d].$$
This proves the proposition.\end{proof}

\begin{remark}Note that for a smooth connected affine curve $C=\Spec B$ an analogue of Proposition \ref{prop:smooth_affine_var_neighborhood_infty} (with the same proof) provides a short exact sequence
$$0\to B\to \hhat{B}_{\infty}\xto{r} (\Omega^1_B)^*\to 0.$$ Explicitly, for a smooth completion $\bbar{C}\supset C,$ an element $f=\{f_p\in \hhat{K}_{\bbar{C},p}\}_{p\in\bbar{C}-C}\in\hhat{B}_{\infty},$ and a $1$-form $\omega\in\Omega^1_B,$ we have $r(f)(\omega)=\sum\limits_{p\in\bbar{C}-C}\res_p(f_p\omega).$\end{remark}

\section{Example: affine plane, line bundles, algebraizability and algebraicity}
\label{sec:affine_plane}

A very interesting example of our construction is the case of line bundles on the punctured neighborhood of infinity for an affine plane $\A_{\mk}^2.$ The most natural compactification is of course $\PP_{\mk}^2\supset\A_{\mk}^2.$ 

We denote by $(z_0:z_1:z_2)$ the homogeneous coordinates on $\PP_{\mk}^2.$ Let us put $U_i:=\{z_i\ne 0\}\subset\PP_{\mk}^2,$ and $H_i:=\{z_i=0\}\subset\PP_{\mk}^2.$ We identify $\A_{\mk}^2$ with $U_0.$ We put $x:=\frac{z_1}{z_0},$ $y:=\frac{z_2}{z_0}$ -- the coordinates on $\A_{\mk}^2.$ 

Our first goal is to describe the Picard group of the formal scheme $\hhat{(\PP_{\mk}^2)}_{H_0}.$

\begin{prop}\label{prop:Pic_formal_neighborhood}We have a natural isomorphism $\Pic(\hhat{(\PP_{\mk}^2)}_{H_0})\cong \Z\times G,$ where 
\begin{equation}\label{eq:group_G}G=1+x^{-1}y^{-1}\mk[[x^{-1},y^{-1}]]\subset (\mk[[x^{-1},y^{-1}]])^{\times}\end{equation} is the subgroup of the multiplicative group of invertible formal power series in $x^{-1},y^{-1}.$\end{prop}

\begin{proof}Let us note that for $i=1,2$ we have $\Pic(\hhat{(U_i)}_{H_0\cap U_i})=\{1\}.$ Thus, all line bundles on $\hhat{(\PP_{\mk}^2)}_{H_0}$ are trivial on the affine charts, and any line bundle can be given by a cocycle in $\cO(\hhat{(U_{12})}_{H_0\cap U_{12}})^{\times},$ where $U_{12}=U_1\cap U_2.$ One easily computes
$$\cO(\hhat{(U_{1})}_{H_0\cap U_{1}})\cong\mk\left[\frac{y}{x}\right][[x^{-1}]],\quad \cO(\hhat{(U_{2})}_{H_0\cap U_{2}})\cong\mk\left[\frac{x}{y}\right][[y^{-1}]],$$
$$\cO(\hhat{(U_{12})}_{H_0\cap U_{12}})\cong\mk\left[\left(\frac{x}{y}\right)^{\pm 1}\right][[x^{-1}]]\cong \mk\left[\left(\frac{x}{y}\right)^{\pm 1}\right][[y^{-1}]].$$
Passing to invertible elements, we obtain the following multiplicative groups of formal power series
$$G_1:=\cO(\hhat{(U_{1})}_{H_0\cap U_{1}})^{\times}=\mk^{\times}+x^{-1}\mk\left[\frac{y}{x}\right][[x^{-1}]],$$
$$G_2:=\cO(\hhat{(U_{2})}_{H_0\cap U_{2}})^{\times}=\mk^{\times}+y^{-1}\mk\left[\frac{x}{y}\right][[y^{-1}]],$$
$$G_{12}:=\cO(\hhat{(U_{12})}_{H_0\cap U_{12}})^{\times}=\mk^{\times}\left(\frac{x}{y}\right)^{\Z}+x^{-1}\mk\left[\left(\frac{x}{y}\right)^{\pm 1}\right][[x^{-1}]].$$
The restriction maps $f_i:G_i\to G_{12}$ are given by tautological inclusions of groups of invertible power series. We also have a natural homomorphism $f:\Z\times G\to G_{12}$ sending a power series in $G$ to itself, and an integer $n\in\Z$ to $(\frac{x}{y})^n$. We are left to show that the composition homomorphism
\begin{equation}\label{eq:identification_of_Pic}\Z\times G\xto{f} G_{12}\to \coker(G_1\times G_2\xto{(\frac1{f_1},f_2)}G_{12})\end{equation} is an isomorphism.

For that, it is convenient to introduce decreasing filtrations $F^{\bullet}G_1,$ $F^{\bullet}G_2,$ $F^{\bullet}G_{12},$ and $F^{\bullet}(\Z\times G).$ We put $$F^0 G_{12}:=G_{12},\quad F^nG_{12}:=1+x^{-n}\mk\left[\left(\frac{x}{y}\right)^{\pm 1}\right][[x^{-1}]]\quad\text{for}\quad n>0,$$ and for $G_1,G_2$ and $\Z\times G$ the filtrations are induced via injective homomorphisms $f_1,$ $f_2$ and $f.$ Also, for convenience we define a filtration on $\mk^{\times}$ by $F^0\mk^{\times}=\mk^{\times},$ $F^{>0}\mk^{\times}=0.$ It is straightforward to check that the complex of commutative groups
\begin{equation}\label{eq:acyclic_complex_comm_groups}\{1\}\to \mk^{\times}\to G_1\times G_2\times (\Z\times G)\xto{(f_1,\frac1{f_2},f)}G_{12}\to\{1\}\end{equation} becomes acyclic after applying $\gr^n_F,$ $n\geq 0.$ Since for each of the groups the filtration is complete, the complex \eqref{eq:acyclic_complex_comm_groups} is acyclic. Therefore, the map \eqref{eq:identification_of_Pic} is an isomorphism. This proves the proposition.
\end{proof}

\begin{cor}\label{cor:Pic_punctured_neighborhood} We have an isomorphism $\Pic(\hhat{(\PP_{\mk}^2)}_{H_0}-H_0)\cong G,$ where $G$ is introduced in \eqref{eq:group_G}.\end{cor}

For each formal power series $g\in G,$ we denote by $L_g$ the corresponding line bundle on $\hhat{(\PP_{\mk}^2)}_{H_0}-H_0.$ Our next goal is to describe a perfect almost DG $\mk[x,y]$-module corresponding to $L_g$ under the equivalence $\Perf(\hhat{(\PP_{\mk}^2)}_{H_0}-H_0)\simeq\Perf_{top}(\hhat{\mk[x,y]}_{\infty}).$ 

Let us denote by $V$ a countable-dimensional vector space over $\mk$ with the basis $e_{ij},$ $i,j\geq 0$ (we consider $V$ as a complex concentrated in degree zero). Let $g=1+\sum\limits_{i,j=1}^{\infty}\lambda_{ij}x^{-i}y^{-j}\in G,$ where $\lambda_{ij}\in\mk.$ We define a homomorphism $\rho_g:\mk[x,y]\to\End_{\Calk_{\mk}}(V)$ by setting $\rho_g(x)$ and $\rho_g(y)$ to be the projections of actual linear operators $\Phi_{g,x},\Phi_{g,y}:V\to V,$ where $$\Phi_{g,x}(e_{ij})=e_{i+1,j}-\sum\limits_{l=1}^j \lambda_{i+1,l}e_{0,j-l},\quad \Phi_{g,y}(e_{ij})=e_{i,j+1}.$$ Note that $[\Phi_{g,x},\Phi_{g,y}](e_{ij})=-\lambda_{i+1,j+1}\cdot e_{00},$ hence $\rk[\Phi_{g,x},\Phi_{g,y}]\leq 1<\infty$ and the homomorphism $\rho_g$ is well defined. We denote by $M_g\in\bR\un{\Hom}(\mk[x,y]^{op},\Calk_{\mk})$ the almost DG module given by $\rho_g.$ 

\begin{prop}\label{prop:M_g_and_L_g}The almost DG $\mk[x,y]$-module $M_g$ is perfect, and it corresponds to $L_g$ under the equivalence $\Perf(\hhat{(\PP_{\mk}^2)}_{H_0}-H_0)\simeq\Perf_{top}(\hhat{\mk[x,y]}_{\infty}).$\end{prop}

\begin{proof}We denote by $\cL_g\in\Pic(\hhat{(\PP_{\mk}^2)}_{H_0})$ the line bundle corresponding to $(0,g)\in\Z\times G.$ Let us denote by $\cF_g\in\cT_{H_0}$ the object corresponding to $\cL_g$ under the equivalence of Theorem \ref{th:perf_of_formal_is_T_Z}. In order to identify the object of $\Perf_{top}(\hhat{\mk[x,y]}_{\infty})$ corresponding to $L_g,$ it suffices to compute the complex of vector spaces $\bR\Gamma(\PP_{\mk}^2,\cF_g)[1],$ and describe the almost $\mk[x,y]$-action on it. We will do the computations by \v{C}ech.

We easily see by construction of equivalence in Theorem \ref{th:perf_of_formal_is_T_Z} that $$\cF_g(U_1)[1]\cong\mk\left[x^{\pm 1},\frac{y}{x}\right]/\mk\left[x^{-1},\frac{y}{x}\right]=:R_1,\quad \cF_g(U_2)[1]\cong\mk\left[y^{\pm 1},\frac{x}{y}\right]/\mk\left[y^{-1},\frac{x}{y}\right]=:R_2,$$
$$\cF_g(U_{12})[1]\cong \mk[x^{\pm 1},y^{\pm 1}]/\mk\left[x^{-1},\left(\frac{x}{y}\right)^{\pm 1}\right]=:R_{12}.$$ We have natural inclusions $r_i:R_i\to R_{12},$ and the restriction morphisms $\cF_g(U_i)\to\cF_g(U_{12})$ are given by $g r_1$ and  $r_2.$ Thus, the complex $\bR\Gamma(\PP_{\mk}^2,\cF_g)[1]$ is quasi-isomorphic to the following complex (placed in degrees $0,1$).
$$\cK^{\bullet}=\{R_1\oplus R_2\xto{(gr_1,-r_2)}R_{12}\}.$$
For convenience we introduce increasing (exhausting) filtrations $F_{\bullet}R_1,$ $F_{\bullet}R_2,$ $F_{\bullet}R_{12}.$ We put $F_0 R_{12}:=0,$ $F_n R_{12}:=x^n\mk\left[x^{-1},\frac{y}{x}\right]/\mk\left[x^{-1},\frac{y}{x}\right]$ for $n>0,$ and the filtrations on $R_i$ are induced by the inclusions $r_i.$ Note that the action of $g$ preserves the filtration on $R_{12}$ and induces the identity maps on the $\gr_n^F,$ $n>0.$ It is easy to see that the map $(gr_1,-r_2)$ induces surjections on $\gr_{\bullet}^F,$ hence it is surjective and $H^1(\cK^{\bullet})=0.$ 

Let us introduce the subspace $V'\subset R_1$ spanned by monomials $x^iy^j,$ where $i,j\geq 0,$ and $i+j>0.$ We have a splitting $\pr:R_1\to V'$ vanishing on monomials $x^iy^j$ with $i<0.$ This splitting of $R_1$ is compatible with filtrations, and we have the induced filtration $F_{\bullet}V'$ (by the degree of monomials). The composition $u:H^0(\cK^{\bullet})\to R_1\xto{\pr}V'$ induces isomorphisms on the $\gr_n^F,$ hence it is an isomorphism.

We now compute the almost action of $\mk[x,y]$ on $V'.$ For that we first compute $\bR\Hom(\cO(-H_0),\cF_g)[1]\cong\bR\Gamma(\cF_g(H_0))[1].$ We have natural identifications $\cF_g(H_0)(U_i)=R_i/F_1R_i,$ and $\cF_g(H_0)(U_{12})=R_{12}/F_1R_{12}.$ We see that the complex $\bR\Gamma(\cF_g(H_0))[1]$ is naturally quasi-isomorphic to
$$\cK'^{\bullet}:=\cK^{\bullet}/F_1\cK^{\bullet}=\{R_1/F_1R_1\oplus R_2/F_1R_2\xto{(gr_1,-r_2)}R_{12}/F_1R_{12}\}.$$
The natural map $\bR\Gamma(\cF_g)[1]\to\bR\Gamma(\cF_g(H_0))[1]$ (induced by the inclusion $\cO\hto\cO(H_0)$) corresponds to the projection $\cK^{\bullet}\to\cK'^{\bullet}.$ As above we verify that $H^1(\cK'^{\bullet})=0,$ and obtain an isomorphism $u':H^0(\cK'^{\bullet})\to R_1/F_1R_1\xto{\pr'}V'',$ where $V''\subset V'$ is the codimension two subspace spanned by $x^iy^j$ with $i,j\geq 0,$ $i+j\geq 2,$ and the splitting $\pr'$ is defined in the same way.

The variables $x,y$ provide the sections of $\cO(H_0),$ hence they induce the maps $x,y:\bR\Gamma(\cF_g)[1]\to\bR\Gamma(\cF_g(H_0))[1].$ These maps correspond to actual morphisms of complexes $\cK^{\bullet}\to\cK'^{\bullet},$ given by the well-defined multiplications by $x,$ $y.$ We need to compute the compositions
$$T_x:V'\xto{u^{-1}}H^0(\cK^{\bullet})\xto{x}H^0(\cK'^{\bullet})\xto{u'}V'',\quad T_y:V'\xto{u^{-1}}H^0(\cK^{\bullet})\xto{y}H^0(\cK'^{\bullet})\xto{u'}V''.$$ First, we have a commutative diagram 
$$
\xymatrix{R_1\ar[r]^{\pr}\ar[d]^{y} & V'\ar[d]^{y}\\
R_1/F_1R_1\ar[r]^{\pr'} & V''.}$$
It follows that $T_y(x^iy^j)=x^iy^{j+1}$ for $i,j\geq 0,$ $i+j\geq 1.$

To compute $T_x,$ let us consider the element $u^{-1}(x^py^q)=x^py^q+\sum\limits_{i<0,j\geq 1-i}\mu_{ij}x^iy^j$ (the sum is finite). Also, we denote the coefficient of $x^iy^j$ in $g$ by $\lambda_{ij},$ $i,j>0.$ Note that $gr_1(u^{-1}(x^py^q))\in r_2(R_2).$ Vanishing of the coefficients of $x^{-1}y^n$ ($n\geq 2$) for the element $gr_1(u^{-1}(x^py^q))$ implies $$\mu_{-1,n}=\begin{cases}-\lambda_{p+1,q-n} & \text{for }2\leq n\leq q-1;\\
0 & \text{for }n\geq q.\end{cases}$$
It follows that $T_x(x^py^q)=x^{p+1}y^q-\sum\limits_{i=1}^{q-2}\lambda_{p+1,i}y^{q-i}.$ 

We can treat the operators $T_x,T_y:V'\to V''$ as endomorphisms of $V'$ by composing with embedding $V''\to V'.$ They define the desired almost action of $\mk[x,y]$ on $V'.$ The map $\psi:V'\to V,$ $\psi(x^py^q)=e_{p,q},$ is compatible with this almost action on $V'$ and the almost action $\rho_g$ on $V.$ Finally, $\psi$ is an almost isomorphism (since $\Ker(\psi)=0,$ $\dim\coker(\psi)=1$). This proves the proposition.  
\end{proof}

It seems natural to ask which of the line bundles $L_g$ (or, equivalently, perfect almost modules $M_g$) are algebraizable perfect complexes. Below we answer this question in the special case when $g\in 1+x^{-1}y^{-1}\mk[[y^{-1}]].$ It is convenient to introduce the following auxiliary almost DG $\mk[x,y]$-modules. We denote by $W$ the countable-dimensional vector space with the basis $\{e_i\}_{i\geq 0}.$ Take any Laurent power series $h=\sum\limits_{j\leq d}\mu_jy^j\in\mk((y^{-1})).$ We define the homomorphism $\rho'_h:\mk[x,y]\to\End_{\Calk_{\mk}}(W)$ by setting $\rho'_h(x),\rho'_h(y)$ to be the projections of actual linear operators $\Psi_{h,x},\Psi_{h,y},$ where $$\Psi_{h,x}(e_i)=\sum\limits_{j=-i}^d \mu_je_{i+j},\quad \Psi_{h,y}(e_i)=e_{i+1}.$$ Again $\rk[\Psi_{h,x},\Psi_{h,y}]\leq 1<\infty,$ since $[\Psi_{h,x},\Psi_{h,y}](e_i)=\mu_{-i-1}e_0.$ We denote by $N_h\in \bR\un{\Hom}(\mk[x,y]^{op},\Calk_{\mk})$ the almost DG $\mk[x,y]$-module given by $\rho'_h.$

\begin{prop}\label{prop:properties_of_N_h}1) For any $h\in y^{-1}\mk[[y^{-1}]],$ we have an exact triangle $$M_{1-x^{-1}h}\to\hhat{\mk[x,y]}_{\infty}\to N_h\to M_{1-x^{-1}h}[1]$$ in $[\bR\un{\Hom}(\mk[x,y]^{op},\Calk_k)].$

2) For any $h\in \mk((y^{-1}))$ the almost $\mk[x,y]$-module $N_h$ is perfect.\end{prop}

\begin{proof}1) We have a short exact sequence of vector spaces $$0\to \mk[x,y]\to V \to W\to 0,$$
where the maps are given by $x^iy^j\mapsto e_{i+1,j},$ and $e_{0j}\mapsto e_j,$ and $e_{ij}\mapsto 0$ for $i>0.$ Moreover, the maps in the short exact sequence are compatible with the almost actions of $\mk[x,y]:$ on itself by multiplication, on $V$ by $\rho_{1-x^{-1}h},$ and on $W$ by $\rho'_h.$ The assertion follows.

2) Applying a polynomial change of variables of the form $(x,y)\mapsto (x+P(y),y)$ if necessary, we may assume that $h\in y^{-1}\mk[[y^{-1}]].$ In this case the perfectness follows from 1) and Proposition \ref{prop:M_g_and_L_g}.\end{proof}

For convenience, from now on in this section we put $A:=\mk[x,y].$

\begin{lemma}\label{lem:computation_of_Rhom's_divisor}We have $\bR\Hom(\hhat{A}_{\infty},N_h)\cong \mk((y^{-1})),$ $\bR\Hom(N_h,\hhat{A}_{\infty})\cong\mk((y^{-1})),$ and in both cases the $A$-module structure on $\mk((y^{-1}))$ is induced by a homomorphism $\varphi_h:A\to \mk((y^{-1})),$ given by $\varphi_h(x)=h,$ $\varphi_h(y)=y.$\end{lemma}

\begin{proof}Arguing as in the proof of Proposition \ref{prop:properties_of_N_h} 2), we reduce to the case $h\in \mk[[y^{-1}]].$ It is convenient to compute the $\bR\Hom$ complexes in the geometric framework. We define a quasi-coherent sheaf $\cG_h\in\cT_{H_0}$ such that the projection of $\cG_h[-1]$ to $\cT_{H_0}/\Perf_{H_0}(\PP_{\mk}^2)$ will correspond to $N_h.$ The sheaf $\cG_h$ is supported at the point $p_0=(0:0:1).$ The local coordinates at $p_0$ are given by $y^{-1},\frac{x}{y}.$ The vector space $\Gamma(\PP^2_k,\cG_h)$ is isomorphic as a $\mk[y^{-1}]$-module to $\mk((y^{-1}))/\mk[[y^{-1}]].$ The function $\frac{x}{y}$ is acting by $y^{-1}h.$ Since both actions of $y^{-1}$ and $\frac{x}{y}$ are locally nilpotent, the sheaf $\cG_h$ is well-defined. It is easy to see that $\cG_h$ is indeed an object of $\cT_{H_0},$ and the almost $A$-module $\Gamma(\PP^2,\cG_h)$ is isomorphic to $N_h$ via $y^{i}\mapsto e_{i-1},$ $i>0.$

Since $\supp(\cG_h)\subset U_2,$ we may compute the RHom's in the category $\Perf(\hhat{U_2}_{H_0\cap U_2}-(H_0\cap U_2))\simeq\Perf(B),$ where $B:=\mk[\frac{x}{y}]((y^{-1})).$ The complex $\cG_h[-1]$ corresponds to the $B$-module $B/(\frac{x}{y}-y^{-1}h).$ We deduce the isomorphisms of $\mk[x,y]$-modules
$$\bR\Hom(\hhat{A}_{\infty},N_h)\cong\bR\Hom_{B}(B,B/(\frac{x}{y}-y^{-1}h))\cong B/(\frac{x}{y}-y^{-1}h),$$
$$\bR\Hom(N_h,\hhat{A}_{\infty})\cong\bR\Hom_{B}(B/(\frac{x}{y}-y^{-1}h),B)\cong B/(\frac{x}{y}-y^{-1}h)[-1].$$To finish the proof, it suffices to note that the natural map $\mk((y^{-1}))\to B/(\frac{x}{y}-y^{-1}h)$ is an isomorphism.\end{proof}

The following result describes when the almost $\mk[x,y]$-module $N_h$ is an algebraizable perfect complex. We write "$\mk((y^{-1}))_{\varphi_h}$" for the $A$-module structure on $\mk((y^{-1}))$ coming from $\varphi_h.$

\begin{theo}\label{th:algebraizable_algebraic}For a Laurent power series $h\in\mk((y^{-1})),$ the following are equivalent:

$\rm{(i)}$ the object $N_h\in\Perf(\hhat{\mk[x,y]}_{\infty})$ is algebraizable;

$\rm{(ii)}$ the power series $h$ is algebraic (that is, algebraic over a subfield $\mk(y)\subset\mk((y^{-1}))$).

If in addition $h\in y^{-1}\mk[[y]]$ then both $\mathrm{(i)}$ and $\mathrm{(ii)}$ are equivalent to

$\rm{(iii)}$ the line bundle $L_{1-x^{-1}h}\in\Pic(\hhat{(\PP^2_{\mk})}_{H_0}-H_0)$ is an algebraizable perfect complex.\end{theo}

\begin{proof}For convenience we identify the $A\mhyphen A$-bimodules $A$ and $\Omega^2_A$ by choosing the $2$-form $dx\wedge dy.$ In particular, we have an isomorphism of graded algebras $H^{\bullet}(\hhat{A}_{\infty})\cong A\oplus A^*[-1].$

(i)$\Rightarrow$(ii) Suppose that $h$ is transcendental. We need to show that $N_h$ is not algebraizable. This is equivalent to the following: the morphism 
\begin{equation}\label{eq:non-isomorphism}\bR\Hom(\hhat{A}_{\infty},N_h)\Ltens{\hhat{A}_{\infty}}\bR\Hom(N_h,\hhat{A}_{\infty})\to \bR\Hom(N_h,N_h)\end{equation}
is not an isomorphism. 

The crucial observation for the proof is that for $h$ being transcendental, the $A$-module $\mk((y^{-1}))_{\varphi_h}$ is actually a $\mk(x,y)$-vector space. Indeed, in this case the homomorphism $\varphi_h$ is injective, and since $\mk((y^{-1}))$ is a field, $\varphi_h$ extends to a homomorphism of fields $\mk(x,y)\to \mk((y^{-1})).$

Let us note that for $\mk(x,y)$-vector space $L,$ the following holds:

$\bullet$ $L$ is flat over $A;$

$\bullet$ for any $A$-module $N$ the tensor product $L\otimes_A N$ is again a $\mk(x,y)$-vector space.

It suffices to prove that the source of \eqref{eq:non-isomorphism} is concentrated cohomologically in degree $1.$ Indeed, the identity endomorphism of $N_h$ provides a non-zero element of $H^0$ of the target. To deal with the derived tensor product we will need the following standard spectral sequence.

For any $DG$ algebra $B$ over $\mk,$ and for any DG modules $M_1\in D(B),$ $M_2\in D(B^{op}),$ we have a converging spectral sequence $$E_2^{p,q}=\Tor_{-p}^{H^{\bullet}(B)}(H^{\bullet}(M_1),H^{\bullet}(M_2))^{q}\Rightarrow H^{p+q}(M_1\Ltens{B}M_2).$$ Here $\Tor$-spaces between graded $H^{\bullet}(B)$-modules are considered as graded vector spaces, and in the formula for $E_2^{p,q}$ we take the $q$-th graded component of the $(-p)$-th $\Tor$-space.

In our situation, in view of Lemma \ref{lem:computation_of_Rhom's_divisor}, it suffices to show that \begin{equation}\label{eq:only_zeroth_cohomology}H^{\ne 0}(\mk((y^{-1}))_{\varphi_h}\Ltens{H^{\bullet}(\hhat{A}_{\infty})}\mk((y^{-1}))_{\varphi_h})=0\end{equation} (here we interpret $H^{\bullet}(\hhat{A}_{\infty})$ as a DG algebra with zero differential, and consider the derived tensor product of DG modules concentrated in degree zero). Since $H^{\bullet}(\hhat{A}_{\infty})$ is a trivial square-zero exnetsion of $A$ by $A^*[-1],$ we have a natural isomorphism
\begin{equation}\label{eq:tensor_product_over_cohomology}\mk((y^{-1}))_{\varphi_h}\Ltens{H^{\bullet}(\hhat{A}_{\infty})}\mk((y^{-1}))_{\varphi_h}\cong\bigoplus\limits_{n=0}^{\infty}\mk((y^{-1}))_{\varphi_h}\Ltens{A}\underbrace{A^*\Ltens{A}\dots\Ltens{A}A^*}_{n\text{ times}}\Ltens{A}\mk((y^{-1}))_{\varphi_h}.\end{equation}
By the above discussion on $\mk(x,y)$-vector spaces, each direct summand in the RHS of \eqref{eq:tensor_product_over_cohomology} is isomorphic to the non-derived tensor product, hence its only non-zero cohomology is $H^0.$ This proves \eqref{eq:only_zeroth_cohomology} and the implication (i)$\Rightarrow$(ii).

(ii)$\Rightarrow$(i) Let $h\in\mk((y^{-1}))$ be an algebraic Laurent power series.
Let $Q=x^n+f_1(y)x^{n-1}+\dots+f_0(y)\in\mk(y)[x]$ be the minimal monic polynomial of $h$ over $\mk(y).$
We denote by $S\in\mk[y]$ the minimal polynomial in $y$ such that $P=SQ\in\mk[x,y]\subset \mk(x,y).$ Clearly, $P$ is an irreducible polynomial. We put $d:=\deg(P),$ and $\tilde{P}:=z_0^d P(\frac{z_1}{z_0},\frac{z_2}{z_0}).$ Then $X:=\{\tilde{P}=0\}\subset \PP_{\mk}^2$ is an integral complete curve, and $Y:=\Spec(\mk[x,y]/(P))\cong X\cap U_0\subset X$ is its open affine subscheme. We denote by $\tilde{X}$ the normalization of $X$ at $X-Y,$ so that we still have an open embedding $Y\hto\tilde{X}.$ Applying Proposition \ref{prop:affine_curve_almost_modules} to $\tilde{X}$ and $Y,$ and restricting the scalars from $\cO(Y)$ to $\mk[x,y],$ we obtain an isomorphism
of almost $\mk[x,y]$-modules
$$\mk[x,y]/(P)\xto{\varphi}\bigoplus_{p\in \tilde{X}-Y}\mk(\tilde{X})/\cO_{\tilde{X},p}.$$ In particular, each of the direct summands is an algebraizable perfect almost $\mk[x,y]$-module.

We claim that there is a point $p_h\in\tilde{X}-Y,$ such that the almost $\mk[x,y]$-module $\mk(\tilde{X})/\cO_{\tilde{X},p_h}$ is isomorphic to $N_h.$ This assertion implies the algebraizability of $N_h.$

The claim is proved as follows. Consider a map $f:\tilde{X}\to\PP_{\mk}^1$ given by the rational function $y.$ The fiber $f^{-1}(\infty)$ is contained in $\tilde{X}-Y,$ hence it consists of regular points. Furthermore, the points of $f^{-1}(\infty)$ bijectively correspond to irreducible factors of $Q$ in $\mk((y^{-1}))[x].$ One of these factors equals $x-h,$ and we set $p_h\in f^{-1}(\infty)$ to be the corresponding point. We have an isomorphism of $\mk[x,y]$-modules $\hhat{K}_{\tilde{X},p_h}\cong \mk((y^{-1}))[x]/(x-h).$ The subspace $\hhat{\cO}_{\tilde{X},p_h}$ identifies with $\im(\mk[[y^{-1}]]\hto\mk((y^{-1}))[x]/(x-h)).$ We define an isomorphism of vector spaces $W\xto{\sim}\hhat{K}_{\tilde{X},p_h}/\hhat{\cO}_{\tilde{X},p_h}\cong \mk(\tilde{X})/\cO_{\tilde{X},p_h}$ by sending $e_i$ to $y^{i+1},$ $i\geq 0.$ It is easy to see that this isomorphism is compatible with the almost $\mk[x,y]$-action $\rho'_h$ on $W$ and the almost $\mk[x,y]$-action on $\mk(\tilde{X})/\cO_{\tilde{X},p_h}.$ This proves the implication (ii)$\Rightarrow$(i).   

Finally, assuming that $h\in y^{-1}\mk[[y^{-1}]],$ we notice that $\rm{(i)}\Leftrightarrow\rm{(iii)}.$ Indeed, by Proposition \ref{prop:M_g_and_L_g}, algebraizability of $L_{1-x^{-1}h}$ is equivalent to that of $M_{1-x^{-1}h}.$ By Proposition \ref{prop:properties_of_N_h} 1), the algebraizability of $M_{1-x^{-1}h}$ is equivalent to that of $N_{h}.$ This proves the theorem. 
\end{proof}



It seems plausible that for a general $g\in G$ the algebraizability of $L_g$ as a perfect complex is equivalent to algebraicity of $g$ over $\mk(x,y).$ We will address this question elsewhere.

\section{Example: derived category of coherent sheaves on a proper singular scheme}
\label{sec:neighborhood_for_D^b_coh}

In this section $\mk$ is a perfect base field, and $X$ is a proper scheme over $\mk.$ We are interested in the DG category  
$\cB=\mD^b_{coh}(X)$ and its formal punctured neighborhood of infinity $\hhat{\mD^b_{coh}(X)}_{\infty}.$ By \cite[Theorem 6.3]{Lu}, $\mD^b_{coh}(X)$ is smooth.

We recall the notion of an $!\mhyphen$perfect complex.
\begin{defi}Let $Y$ be a separated scheme of finite type over $\mk.$ We denote by $\bbD_Y\in D^b_{coh}(Y)$ the dualizing complex. An object $\cE\in D^b_{coh}(Y)$ is $!$-perfect if the complex $\bR\cHom_{\cO_Y}(\cF,\D_Y)$ is perfect. We denote by $D_{!\mhyphen\perf}(Y)\subset D^b_{coh}(Y)$ (resp. $!\mhyphen Perf(Y)\subset \mD^b_{coh}(Y)$) the full (DG) subcategory formed by $!$-perfect complexes.\end{defi}

By Proposition \ref{prop:pseudo-perfect_over_coherent}, we have $\PsPerf(\mD^b_{coh}(X))\simeq !\mhyphen\Perf(X)\subset \mD^b_{coh}(X).$ By Proposition \ref{prop:pseudo-perfect_in_the_kernel}, we have a functor $\mD_{sg}(X)^{op}\simeq \mD^b_{coh}(X)/!\mhyphen\Perf(X)\to \hhat{\mD^b_{coh}(X)}_{\infty}.$ The main result of this section is the following.

\begin{theo}\label{th:neighborhood_of_infty_for_D^b_coh}The functor $\mD_{sg}(X)^{op}\to \hhat{\mD^b_{coh}(X)}_{\infty}$ is a quasi-equivalence.\end{theo}

We need the following general result, which is of independent interest. First, we formulate a condition on a locally proper DG category $\cC.$

$(*):$ for any two pseudo-perfect $\cC$-modules $M,N$ the natural morphism $M\Ltens{\cC}N^*\to\bR\Hom_{\cC}(M,N)^*$ is an isomorphism. 

\begin{prop}\label{prop:Koszul_duality_and_neighborhoods}Let $\cB$ be a smooth DG category, and $\cC$ a locally proper DG category. Let $M\in\cC\mhyphen\Mod\mhyphen\cB$ be a bimodule satisfying the following conditions:

(i) $M\in D_{\pspe}(\cB\otimes\cC^{op});$

(ii) the functor $-\Ltens{\cC}M:D_{perf}(\cC)\to D_{\pspe}(\cB)$ is fully faithful;

(iii) the functor $M\Ltens{\cB}-:D_{perf}(\cB^{op})\to D_{\pspe}(\cC^{op})$ is fully faithful;

(iv) the category $\cC^{op}$ satisfies the condition $(*).$

Then the functor from (ii) is an equivalence and the functor $\cB/\PsPerf(\cB)\to\hhat{\cB}_{\infty}$ (obtained from Proposition \ref{prop:pseudo-perfect_in_the_kernel}) is a quasi-equivalence.
\end{prop}

\begin{proof}We may and will assume that $\cC\subset\cB$ is a full DG subcategory and the bimodule $M$ is given by $M(x,y)=\cB(y,\iota(x)),$ for $x\in\cC,$ $y\in\cB$ and $\iota:\cC\to\cB$ the inclusion functor. It suffices to prove that the natural functor $\cB/\cC\to \hhat{\cB}_{\infty}$ is a quasi-equivalence.

Clearly, this functor is essentially surjective, so the only issue is quasi-fully-faithfulness. Let us take any $x,y\in\cB.$ We have
$$(\cB/\cC)(x,y)\cong Cone(\cB(\iota(-),y)\Ltens{\cC}\cB(x,\iota(-))\to\cB(x,y))\quad\text{in }D(\mk).$$
By Proposition \ref{prop:neighborhood_of_infty_Hochschild_cochains},
$$\hhat{\cB}_{\infty}(\bbar{\bY}(x),\bbar{\bY}(y))\cong Cone(C^{\bullet}(\cB^{op},\bY(x)^*\otimes \bY(y))\to \cB(x,y))\quad\text{in }D(\mk).$$
It suffices to figure out what is the natural morphism
\begin{equation}\label{eq:key_morphism}\cB(\iota(-),y)\Ltens{\cC}\cB(x,\iota(-))\to C^{\bullet}(\cB^{op},\bY(x)^*\otimes \bY(y))\end{equation}
and then prove that it is an isomorphism in $D(k).$ The morphism \eqref{eq:key_morphism} is the following composition:
\begin{multline*}\cB(\iota(-),y)\Ltens{\cC}\cB(x,\iota(-))\xto{\sim}\bR\Hom_{\cB}(M,\bY(y))\Ltens{\cC}M_x\xto{\sim}C^{\bullet}(\cB^{op},M^*\otimes \bY(y))\Ltens{\cC}M_x\\
\xto{f_1} C^{\bullet}(\cB^{op},M^*\Ltens{\cC}M_x\otimes \bY(y))\xto{f_2} C^{\bullet}(\cB^{op},\bR\Hom_{\cC^{op}}(M_x,M)^*\otimes \bY(y))\xto{f_3} C^{\bullet}(\cB^{op},\bY(x)^*\otimes \bY(y))\end{multline*} 
in $D(\mk).$ Here the morphism $f_1$ is an isomorphism by the smoothness of $\cB$ (i.e. $C^{\bullet}(\cB^{op},-)$ commutes with small direct sums). The morphism $f_2$ is an isomorphism by (iv). The morphism $f_3$ is an isomorphism by (iii). This proves the lemma.  
\end{proof}

\begin{proof}[Proof of Theorem \ref{th:neighborhood_of_infty_for_D^b_coh}]Let us put $\cB=\mD^b_{coh}(X),$ and $\cC:=!\mhyphen\Perf(X).$ We have an inclusion functor $\cC\hto\cB,$ and denote by $M$ the corresponding $\cC\mhyphen\cB$-bimodule. We claim that $\cB,$ $\cC$ and $M$ satisfy the conditions (i)-(iv) of Proposition \ref{prop:Koszul_duality_and_neighborhoods}.

Conditions (i)-(ii) follow from Proposition \ref{prop:pseudo-perfect_over_coherent}. Condition (iii) after applying $\bR\Hom(-,\bbD_X)$ basically states that the functor $D^b_{coh}(X)\to D_{\pspe}(\Perf(X))\simeq D(X)$ is fully faithful, which is clear. Condition (iv) is proved in Appendix \ref{app:boundedness_etc}, see Corollary \ref{cor:*_satisfied_for_proper_schemes}.

Now the assertion follows from Proposition \ref{prop:Koszul_duality_and_neighborhoods}.
\end{proof}

It follows immediately from Theorem \ref{th:neighborhood_of_infty_for_D^b_coh} that $\Perf_{alg}(\hhat{\mD^b_{coh}(X)}_{\infty})\simeq \bbar{\mD_{sg}}(X)^{op}.$ The following conjecture seems to be true.

\begin{conj}For a proper scheme $X$ over a perfect field $\mk,$ we have $\Perf_{top}(\hhat{\mD^b_{coh}(X)}_{\infty})=\Perf_{alg}(\hhat{\mD^b_{coh}(X)}_{\infty}).$\end{conj}

\section{Concluding remarks}
\label{sec:concluding}

We now briefly mention some aspects of the construction $\cB\waveto \Perf_{top}(\hhat{\cB}_{\infty})$ which were not discussed in this paper. We do not give any proofs in this section. The statements and constructions mentioned here are to be addressed elsewhere.

\subsection{Residues and reciprocity}
\label{ssec:residues_etc}

First, let us note that we have a composition functor
$$\Perf(\cB^{op})\otimes\Perf_{top}(\hhat{\cB}_{\infty})\hto \Perf(\cB^{op})\otimes \bR\un{\Hom}(\cB^{op},\Calk_{\mk})\xto{\eval}\Calk_{\mk}.$$
In particular, each object $M\in D_{\perf}(\cB^{op})$ defines residue homomorphisms
$$\res(M):K_i(\Perf_{top}(\hhat{\cB}_{\infty}))\to K_{i-1}(\mk),$$
$$\res(M):HH_i(\Perf_{top}(\hhat{\cB}_{\infty}))\to HH_{i-1}(\mk),$$
and similarly for other localizing invariants.

From the commutative diagram
$$
\xymatrix{\Perf(\cB^{op})\otimes \Perf(\cB)\ar[d]\ar[r] & \Mod_{\mk}\ar[d]\\
\Perf(\cB^{op})\otimes \Perf_{top}(\hhat{\cB}_{\infty})\ar[r] & \Calk_{\mk}}
$$
we obtain the abstract reciprocity laws: the compositions
\begin{equation}\label{eq:reciprocity_K_theory}K_i(\Perf(\cB))\xto{\bbar{\bY}} K_i(\Perf_{top}(\hhat{\cB}_{\infty}))\xto{\res(M)} K_{i-1}(\mk)\end{equation}
and
\begin{equation}\label{eq:reciprocity_HH}HH_i(\Perf(\cB))\xto{\bbar{\bY}}HH_i(\Perf_{top}(\hhat{\cB}_{\infty}))\to HH_{i-1}(\mk)\end{equation}
vanish.

Let $C$ be a smooth connected complete curve over $\mk,$ $\cB=\mk(C),$ and $M=\mk(C).$ 

It is straightforward to show that the composition 
$$\Omega^1_{\A_{C}/\mk}\cong HH_1(\hhat{\mk(C)}_{\infty})\to HH_1(\Perf_{top}(\hhat{\mk(C)}_{\infty}))\xto{\res(\mk(C))} HH_0(\mk)=\mk$$ 
is given by $fdg\mapsto \sum\limits_{p\in C^{cl}}\Tr_{\mk(p)/\mk}(\res_p(f_p d g_p)),$ for $f=(f_p)_{p\in C^{cl}}, g=(g_p)_{p\in C^{cl}}\in \A_C.$ In this case vanishing of  \eqref{eq:reciprocity_HH} gives exactly the residue theorem, essentially as in \cite{Ta}.

Another interesting example is the composition $$K_2^M(\A_C)\to K_2(\hhat{\mk(C)}_{\infty})\to K_2(\Perf_{top}(\hhat{\mk(C)}_{\infty}))\to \mk^{\times}=K_1(\mk).$$ It is straightforward to check that it sends $[f,g]\in K_2^M(\A_C)$ to  
$\prod\limits_{p\in C^{cl}}\Nm_{\mk(p)/\mk}(f_p,g_p)_{\nu_p},$ where $f=(f_p)_{p\in C^{cl}}, g=(g_p)_{p\in C^{cl}}\in \bbI_C=\A_C^{\times}$ (the group of ideles), and
$$(f_p,g_p)_{\nu_p}=(-1)^{\nu_p(f_p)\nu_p(g_p)}\frac{f_p^{\nu_p(g_p)}}{g_p^{\nu_p(f_p)}}\text{ mod }\m_p$$ (the Hilbert symbol). In this case the vanishing of \eqref{eq:reciprocity_K_theory} gives the Weil reciprocity law, essentially as in \cite{ACK}.

\subsection{Other possible invariants and extra structures}
\label{ssec:more_invariants}

Recall the setting of Section \ref{sec:DG_cat_neighborhoods_main_result}: let $\cS\hto \cA\to \cB$ be a Morita short exact sequence of DG categories, such that $\cA$ is smooth and proper. The following general principle seems appropriate to formulate:

\begin{itemize}
\item Suppose that we have some (Morita) invariant $\alpha(T)$ (a DG category, a complex of vector spaces, a mixed complex etc.)  of a locally proper DG category $T.$ Suppose that it does not change (in some sense, e.g. up to canonical isomorphism in a suitable category) under the operation $T\mapsto T/T',$ where $T'\subset T$ is a full smooth and proper DG subcategory (as in Proposition \ref{prop:D_sg_for_gluing_with_smooth_and_proper}). Then there is a natural invariant $\alpha'(\cC)$ which is defined for any smooth DG category $\cC,$ such that in the above notation we have an identification $\alpha(\cS)=\alpha'(\cB).$
\end{itemize}

This principle of course needs some additional restrictions on $\alpha(T),$ e.g. the condition is satisfied for $\mD_{sg}(T),$ but in Section \ref{sec:DG_cat_neighborhoods_main_result} we needed to replace it by $\bbar{\mD_{sg}}(T).$

An example of application of this principle, which is a more or less straightforward generalization of Theorem \ref{th:main_theorem}, is the following quasi-equivalence for any DG category $\mD:$
$$\left(\frac{\bR\un{Hom}(\cS^{op},\cD)}{\Perf(\cS\otimes\cD)}\right)^{\Kar}\simeq \Ker(\bR\un{\Hom}(\cB^{op},\Calk_{\cB})\to\Calk_{\cB\otimes\cD}).$$
A perhaps more surprising example is the following: for any DG category $\cD$ we have a quasi-equivalence
$$\left(\frac{\bR\un{\Hom}(\cD,\cS)}{\PsPerf(\cD^{op})\otimes \cS}\right)^{\Kar}\simeq \Ker(\Perf(\PsPerf(\cD^{op})\otimes\cB)\to \bR\un{\Hom}(\cD,\cB)).$$ 

Finally, we mention the following application of this principle which seems to require a deeper analysis.
We may define a "stable category of locally proper DG categories", enriched over $\Ho_M(\dgcat_{\mk})$ (in fact having a deeper structure as in \cite{Tam}): the objects are locally proper DG categories, and the morphisms are given by
$$\un{\Hom}^{st}(T_1,T_2)=\left(\frac{\bR\un{\Hom}(T_1,T_2)}{\PsPerf(T_1^{op})\otimes T_2}\right)^{\Kar}.$$
The composition is straightforward. It can be shown that for Morita short exact sequences $\cS_1\hto \cA_1\to\cB_1,$ $\cS_2\hto\cA_2\to\cB_2,$ we have a quasi-equivalence:
$$\un{\Hom}^{st}(\cS_1,\cS_2)\simeq \Ker(\Perf(\Perf_{top}(\hhat{\cB_1^{op}}_{\infty})\otimes\cB_2)\to \bR\un{\Hom}(\cB_1^{op},\Calk_{\cB_2}))=:\hhat{\un{\Hom}}(\cB_1,\cB_2).$$  

One can in fact identify the compositions $$\hhat{\un{\Hom}}_{\infty}(\cB_2,\cB_3)\otimes \hhat{\un{\Hom}}_{\infty}(\cB_1,\cB_2)\to \hhat{\un{\Hom}}_{\infty}(\cB_1,\cB_3)$$
(this is not obvious at all from the definition). In this way one can obtain another category enriched over $\Ho_M(\dgcat_{\mk})$ whose objects are smooth DG categories (again, there is a deeper structure as in \cite{Tam}). For example, from this one can get a $E_2$-algebra structure on Hochschild cochains $C^{\bullet}(\cB,\hhat{\cB}_{\infty}),$ which can be thought of as a completion of $C^{\bullet}(\cB)$ at infinity. Also, there is a mixed complex structure on the Hochschild chain complex $C_{\bullet}(\cB,\hhat{\cB}_{\infty})\cong Cone(C_{\bullet}(\cB^{op})^*\xto{\ch(I_{\cB})} C_{\bullet}(\cB)),$ where $\ch(I_{\cB})\in HC^{-}_0(\cB^{op}\otimes \cB)$ is the Chern character of the diagonal bimodule. A pair $(C^{\bullet}(\cB,\hhat{\cB}_{\infty}),C_{\bullet}(\cB,\hhat{\cB}_{\infty}))$ can be equipped with a Tsygan-Tamarkin calculus structure \cite{TT}.

Finally, we mention that for a commutative DG algebra $B,$ which is smooth as an associative DG algebra, one can define a symmetric monoidal structure on $[\Perf_{top}(\hhat{B}_{\infty})],$ so that the functor $[\bbar{\bY}]$ becomes symmetric monoidal.  

\appendix

\section{Perfect complexes on affine formal schemes}
\label{app:formal_affine}

\begin{lemma}\label{lem:ext_of_scalars_complete_algebra}Let $A$ be an associative ring, and $I\subset A$ an ideal such that $A$ is $I$-adically complete.

1) If $I$ is nilpotent, then the functor $-\Ltens{A} A/I:D(A)\to D(A/I)$ is conservative.

2) The functor $-\Ltens{A} A/I:D_{\perf}(A)\to D_{\perf}(A/I)$ is conservative.

3) The extension of scalars functor $-\otimes_A A/I:\Mod\mhyphen A\to \Mod\mhyphen A/I$ induces a full essentially surjective functor on the groupoids of finitely generated projective modules.\end{lemma}

\begin{proof}
1) Conservativity of an exact functor between triangulated categories is equivalent to the vanishing of its kernel. Therefore, by adjunction, it suffices to show that the essential image of the restriction of scalars $D(A/I)\to D(A)$ generates $D(A)$ as a triangulated category. For this, let us take any complex $K^{\bullet}$ of $A$-modules. It has a finite decreasing filtration $F^p=K^{\bullet}\cdot I^p$ (finiteness follows from the nilpotency of $I$). All the subquotients $F^p/F^{p+1}$ are in the essential image of $D(A/I).$ This proves the assertion.

 2) By our assumptions, for any object $M\in D_{\perf}(A)$ we have an isomorphism 
\begin{equation}\label{eq:holim_I_adic}M\cong \holim_n M\Ltens{A}A/I^n.
\end{equation}
 Suppose that for some $M\in\Perf(A)$ we have $M\Ltens{A} A/I=0.$ Then by 1) for all $n\geq 1$ we have $M\Ltens{A} A/I^n=0.$ From \eqref{eq:holim_I_adic} we deduce that $M=0.$ This shows conservativity.
 
3) Let us take some finitely generated projective $A$-modules $P,Q,$ and denote by $\bbar{P},\bbar{Q}$ their extensions of scalars to $A/I.$ Suppose that we have an isomorphism $\bbar{f}:\bbar{P}\xto{\sim}\bbar{Q}.$ Since $P$ is projective, we can lift $\bbar{f}$ to a morphism of $A$-modules $f:P\to Q.$ By 2), $f$ is an isomorphism. This proves fullness.

Let us prove essential surjectivity. Let $\bbar{P}$ be a finitely generated projective $(A/I)$-module. Then there is an idempotent $\bbar{e}\in M_n(A/I),$ such that $\bbar{P}$ is isomorphic to the image of $\bbar{e}:(A/I)^n\to (A/I)^n.$ Since the ring $M_n(A)$ is $M_n(I)$-adically complete, the idempotent $\bbar{e}$ can be lifted to an idempotent $e\in M_n(A).$ Denoting by $P$ the image of $e:A^n\to A^n,$ we see that $P\otimes_A A/I\cong \bbar{P}.$\end{proof}

Let $X$ be a noetherian separated scheme.
Let $\cF\in D_{\perf}(X)$ be any perfect complex. We introduce the following notation. Denote by $b(\cF)$ the maximal integer $m$ such that $\cH^m(\cF)\ne 0.$ Further, put $a(\cF):=-b(\cF^{\vee}).$ We also put $\lambda(\cF):=b(\cF)-a(\cF).$

\begin{prop}\label{prop:inf_extensions_a_F_b_F}Let $i:X\hto \tilde{X}$ be an infinitesimal extension, and $\cF\in D_{\perf}(\tilde{X}).$ Then we have $a(\cF)=a(\bL i^*\cF),$ $b(\cF)=b(\bL i^*\cF).$\end{prop}

\begin{proof}This follows immediately from  Nakayama's lemma.\end{proof}



\begin{prop}\label{prop:when_perfect_is_lf}For a perfect complex $\cF\in\Perf(X)$ and an integer $m\in\Z$ the following are equivalent:

(i) $H^{\ne m}(\cF)=0$ and $H^m(\cF)$ is a locally free sheaf;

(ii) We have $a(\cF)=b(\cF)=m.$\end{prop}

\begin{proof}Evident.\end{proof}


\begin{prop}\label{prop:completion_affine}Let $X=\Spec A$ be affine, and $Z\subset X$ a closed subscheme corresponding to the ideal $I\subset A.$ Then we have $\Perf(\hhat{X}_Z)=\Perf_{alg}(\hhat{X}_Z)\simeq \Perf(\hhat{A}_I).$\end{prop}

\begin{proof}Let $\cF\in\Perf(\hhat{X}_Z)$ be an object given by a sequence $\cF_n\in\Perf(Z_n).$ We prove that $\cF$ is algebraizable by induction on $\lambda(\cF_1).$

For the base of induction, suppose that for some $m\in\Z$ the perfect complex $\cF_1[-m]\in \Perf(Z)$ is isomorphic to a locally free sheaf. Then by Proposition \ref{prop:inf_extensions_a_F_b_F} for all $n\geq 0,$ the object $\cF_n[-m]\in\Perf(Z_n)$ is also isomorphic to a locally free sheaf. Let us denote by $P_n$ the projective $A/I^n\mhyphen$module $\Gamma(Z_n,\cH^m(\cF_n)).$ We have the structural isomorphisms $P_{n+1}\otimes_{A/I^{n+1}}A/I^n\cong P_n.$ By Lemma \ref{lem:ext_of_scalars_complete_algebra} 3), there is a finitely generated projective $\hhat{A}_I$-module $\hhat{P}$ with an isomorphism $\hhat{P}\otimes_{\hhat{A}_I}A/I\cong P_1.$ By Lemma \ref{lem:ext_of_scalars_complete_algebra} 3), we can construct a compatible sequence of isomorphisms $\hhat{P}\otimes_{\hhat{A}_I}A/I^n\cong P_n.$ This shows that $\cF$ is algebraizable.

Now suppose that the assertion is proved for all objects $\cG\in\Perf(\hhat{X}_Z)$ with $\lambda(\cG)\leq l.$ Let us take an object $\cF\in\Perf(\hhat{X}_Z)$ with $\lambda(\cF)=l+1.$ Let us put $M_n:=\Gamma(Z_n,\cH^{b(\cF_n)}(\cF_n)).$ Clearly, each $M_n$ is a finitely generated $A/I^n$-module. By Proposition \ref{prop:inf_extensions_a_F_b_F}, $b(\cF_n)=b(\cF_1)$ for all $n\geq 1,$ hence we have natural isomorphisms $M_{n+1}\otimes_{A/I^{n+1}}A/I^n\cong M_n.$ We take any surjection $\phi_1:(A/I)^d\to M_1,$ and lift it to a compatible sequence of morphisms $\phi_n:(A/I^n)^d\to M_n$ (which are also surjective by Nakayama's lemma). This sequence defines a morphism $\phi:\cO_{\hhat{X}_Z}^d[-m]\to\cF,$ and we have $\lambda(Cone(\phi)_1)\leq l.$ By the induction hypothesis, $Cone(\phi)$ is algebraizable, hence so is $\cF.$\end{proof}

\section{Boundedness, compact approximation and pseudo-perfect DG modules}
\label{app:boundedness_etc}

For simplicity we assume that the base field $\mk$ is perfect.



\begin{prop}\label{prop:pseudo-perfect_over_coherent}Let $X$ be a separated scheme of finite type over $\mk.$ Then the DG category $\PsPerf(\mD^b_{coh}(X))$ is quasi-equivalent to the DG category $!\mhyphen Perf(X)_{\propp}$ of $!\mhyphen$perfect complexes with proper support.\end{prop}

\begin{proof}By the smoothness of $\mD^b_{coh}(X),$ we have an inclusion $D_{\pspe}(\mD^b_{coh}(X))\subset D^b_{coh}(X).$ It suffices to show that for an object $\cF\in D^b_{coh}(X)$ the following are equivalent:

$\rm{(i)}$ For any $\cG\in D^b_{coh}(X)$ we have $\bR\Hom(\cG,\cF)\in D_{\perf}(\mk).$

$\rm{(ii)}$ We have $\cF\in D_{!\mhyphen\perf}(X)_{\propp}.$

After applying the contravariant involution $\bR\cHom_{\cO_X}(-,\bbD_X),$ this equivalence translates into another equivalence of two statements:

$\rm{(i)'}$ For any $\cG\in D^b_{coh}(X)$ we have $\bR\Hom(\cF,\cG)\in D_{\perf}(\mk).$

$\rm{(ii)'}$ We have $\cF\in D_{\perf}(X)_{\propp}.$

The implication $\rm{(ii)'}\Rightarrow\rm{(i)'}$ is clear: for such $\cF$ we have $$\bR\Hom(\cF,\cG)\cong \bR\Gamma(X,\cF^{\vee}\Ltens{\cO_X}\cG)\in D_{\perf}(\mk),$$ since $\cF$ has proper support.

For the implication $\rm{(i)'}\Rightarrow\rm{(ii)'},$ we first suppose that $\cF$ is not perfect. Then for some closed point $x\in X^{cl}$ we have $\bR\Hom(\cF,\cO_x)\not\in D_{\perf}(\mk),$ a contradiction.

Now suppose that $\cF$ is perfect but $\supp(\cF)$ is not proper. By the valuative criterion of properness, there is a closed embedding $i:C\hto \supp(\cF),$ where $C$ is an affine integral curve. Clearly, we have $\bR\Hom(\cF,i_*\cO_C)\not\in D_{\perf}(\mk),$ a contradiction. This proves the implication and the proposition.
\end{proof}

The following definition of compact approximability is a modification of \cite[Definition 8.1]{LO} (for  a single generator), and is motivated by \cite[Theorem 4.1]{LN}.

\begin{defi}\label{def:compact_approx} Let $\cT$ be a compactly generated triangulated category, such that there exist a single compact generator $G\in\cT^c.$ Suppose that we have $\Hom^n(G,G)=0$ for $n>>0$ (this property does not depend on the choice of a generator). An object $E\in\cT$ is called compactly approximable if for any $l\in\Z$ there exists a compact object $F\in\cT^c$ and a morphism $\varphi:F\to E$ which induces isomorphisms $\Hom^i(G,F)\xto{\sim} \Hom^i(G,E)$ for $i>l.$\end{defi}

It will be convenient for us to introduce certain boundedness conditions for DG modules over a DG algebra.

\begin{defi}\label{def:Tor_Proj_bounded}Let $A$ be a DG algebra over $\mk.$ 

1) For $n\in\Z,$ we denote by  $D^{\leq n}(A)\subset D(A)$ (resp. $D^{\geq n}(A)\subset D(A)$) the full subcategory of DG $A$-modules $M$ such that $H^{>n}(M)=0$ (resp. $H^{<n}(M)=0$). We also put $D^-(A):=\bigcup\limits_{m\in\Z}D^{\leq m}(A),$ $D^+(A):=\bigcup\limits_{m\in\Z}D^{\geq m}(A)$ and $D^b(A):=D^-(A)\cap D^+(A).$

2) A DG module $M\in D(A)$ is called "$\Tor$-bounded above" if for some $m\in\Z$ we have $M\Ltens{A} D^{\leq 0}(A^{op})\subset D^{\leq m}(\mk).$

3) A DG module $M\in D(A)$ is called "projectively bounded above" if for some $m\in\Z$ we have $\bR\Hom_A(M,D^{\geq 0}(A))\subset D^{\geq m}(\mk).$\end{defi}

\begin{remark}Definition \ref{def:Tor_Proj_bounded} is essentially invariant under Morita equivalence. Namely, if $\Phi:D(A)\xto{\sim} D(A')$ is an equivalence (given by some bimodule), then there are some integer constants $C_1,$ $C_2,$ $C_3,$ $C_4$ such that $\Phi(D^{\leq n}(A))\subset D^{\leq n+C_1}(A'),$ $\Phi^{-1}(D^{\leq n}(A'))\subset D^{\leq n+C_2}(A),$ $\Phi(D^{\geq n}(A))\subset D^{\geq n+C_3}(A'),$ $\Phi^{-1}(D^{\geq n}(A'))\subset D^{\geq n+C_4}(A).$ In particular, $\Phi(D^-(A))=D^-(A'),$ and similarly for $D^+$ and $D^b.$ Also, a DG module $M\in D(A)$ is $\Tor$-bounded (resp. projectively bounded) above if so is $\Phi(M)\in D(A').$\end{remark}

\begin{prop}\label{prop:proj_bounded_implies_Tor}Let $M$ a be a DG $A$-module. If $M$ is projectively bounded above, then $M$ is $\Tor$-bounded above.\end{prop}

\begin{proof}Indeed, this follows from an isomorphism $$(N_1\Ltens{A}N_2)^*\cong \bR\Hom_A(N_1,N_2^*)$$ for any $N_1\in D(A),$ $N_2\in D(A^{op}).$\end{proof}

\begin{prop}\label{prop:if_diagonal_is_bounded}Let $A$ be a DG $\mk$-algebra.

1) Assume that the diagonal bimodule $A\in D(A\otimes A^{op})$ is $\Tor$-bounded above. Then for some $m\in\Z$ we have $D^{\leq 0}(A)\Ltens{A} D^{\leq 0}(A^{op})\subset D^{\leq m}(\mk).$ In particular, all objects of $D^-(A)$ are $\Tor$-bounded above.

2) Similarly, assume that $A\in D(A\otimes A^{op})$ is projectively bounded above. Then for some $l\in\Z$ we have $\bR\Hom_A(D^{\leq 0}(A),D^{\geq 0}(A))\subset D^{\geq l}(\mk).$\end{prop}

\begin{proof}1) Take any $N\in D^{\leq 0}(A^{op})$ and $M\in D^{\leq 0}(A).$ Choose $m>>0$ such that $A\Ltens{A\otimes A^{op}}D^{\leq 0}(A\otimes A^{op})\subset D^{\leq m}(\mk).$ Then 
$$M\Ltens{A}N\cong A\Ltens{A\otimes A^{op}}(M\tens{\mk}N)\in D^{\leq m}(\mk).$$
This proves the assertion.

2) One argues similarly, using the isomorphism
$$\bR\Hom_A(M,M')\cong\bR\Hom_{A\otimes A^{op}}(A,\Hom_{\mk}(M,M')).$$\end{proof}

\begin{prop}\label{prop:*_satisfied_M_compact_approx}Let $A$ be a proper DG $\mk$-algebra, such that the diagonal $A\mhyphen A$-bimodule is projectively bounded above. Let $M\in D(A)$ be compactly approximable and $N\in D_{\pspe}(A).$ Then the natural map $M\Ltens{A} N^*\to\bR\Hom_A(M,N)^*$ is an isomorphism.\end{prop}

\begin{proof}By Proposition \ref{prop:proj_bounded_implies_Tor}, $A\in D(A\otimes A^{op})$ is also $\Tor$-bounded above. Thus, we may apply Proposition \ref{prop:if_diagonal_is_bounded} and choose the integers $m,l\in\Z$ as in its formulation. We may and will assume that $N\in D^{\geq 0}(A)$ (hence $N^*\in D^{\leq 0}(A^{op})$).

Let $t\in\Z$ be any integer. Since $M$ is compactly approximable, we can find a perfect complex $M_t\in D_{perf}(A)$ and a morphism $f_t:M_t\to M$ such that $Cone(f_t)\in D^{<t}(A).$ Applying Proposition \ref{prop:if_diagonal_is_bounded}, we see that $Cone(f_t\Ltens{A} N^*)\in D^{<(t+m)}(\mk),$ and $Cone(\bR\Hom_A(f_t,N)^*)\in D^{<(t-l)}(\mk).$ From the commutative diagram
$$
\xymatrix{M_t\Ltens{A}N^*\ar[r]^-{\sim}\ar[d]^-{f_t\Ltens{A} N^*} & \bR\Hom(M_t,N)^*\ar[d]^{\bR\Hom_A(f_t,N)^*}\\
M\Ltens{A}N^*\ar[r] & \bR\Hom(M,N)^*}$$
we deduce that the map $H^i(M\Ltens{A} N^*)\to H^i(\bR\Hom_A(M,N)^*)$ is an isomorphism for $i\geq t+\max(m,-l).$ Since $t\in\Z$ can be choosen arbitrarily, we conclude that $M\Ltens{A} N^*\to\bR\Hom_A(M,N)^*$ is an isomorphism.\end{proof}

\begin{prop}\label{prop:box_product_preserves_bounded}Let $A_1,$ $A_2$ be DG algebras, and $M_1\in D(A_1),$ $M_2\in D(A_2).$ If both $M_i$ are projectively bounded above (resp. $\Tor$-bounded above), then so is $M_1\otimes M_2$ over $A_1\otimes A_2.$\end{prop}

\begin{proof}We consider the case when $M_i$ are projectively bounded above, and the case of $\Tor$-bounded above modules is treated analogously.

By definitions, there exists numbers $l_1,l_2\in\Z$ such that $\bR\Hom_{A_i}(M_i,D^{\geq 0}(A_i))\subset D^{\geq l_i}(\mk).$ Now take any DG module $N\in D^{\geq 0}(A_1\otimes A_2).$ Then we have
\begin{multline*}\bR\Hom_{A_1\otimes A_2}(M_1\otimes M_2,N)\cong \bR\Hom_{A_1}(M_1,\bR\Hom_{A_2}(M_2,N))\\
\in\bR\Hom_{A_1}(M_1,D^{\geq l_2}(A_1))\subset D^{\geq l_1+l_2}(\mk).\end{multline*}
Therefore, $M_1\otimes M_2$ is projectively bounded above over $A_1\otimes A_2.$\end{proof}

\begin{prop}\label{prop:*_satisfied_for_all}Let $A$ be a proper DG algebra such that the diagonal $A\mhyphen A$-bimodule is projectively bounded above and compactly approximable. Then for any pseudo-perfect DG $A$-modules $M,N$ we have an isomorphism
$M\Ltens{A} N^*\xto{\sim}\bR\Hom_A(M,N)^*.$\end{prop}

\begin{proof}By Proposition \ref{prop:box_product_preserves_bounded}, the diagonal $(A\otimes A^{op})\mhyphen (A\otimes A^{op})$-bimodule is projectively bounded above. Let us note that 
\begin{equation}\label{eq:reduce_to_diagonal}M\Ltens{A}N^*\cong A\Ltens{A\otimes A^{op}}(M^*\otimes N)^*,\quad \bR\Hom_A(M,N)\cong \bR\Hom_{A\otimes A^{op}}(A,M^*\otimes N).\end{equation}
Applying Proposition \ref{prop:*_satisfied_M_compact_approx} to the DG algebra $A\otimes A^{op},$ and DG $A\otimes A^{op}$-modules $A,$ $M^*\otimes N,$ and taking \eqref{eq:reduce_to_diagonal} into account, we obtain an isomorphism $M\Ltens{A} N^*\xto{\sim}\bR\Hom_A(M,N)^*.$\end{proof}

\begin{prop}\label{prop:pseudo_perfect_comp_approx}Let $A$ be a DG algebra such that the diagonal $A\mhyphen A$-bimodule is $\Tor$-bounded above and compactly approximable. Then any pseudo-perfect DG $A$-module is compactly approximable.\end{prop}

\begin{proof}By Proposition \ref{prop:if_diagonal_is_bounded}, we can find $m\in\Z$ such that $D^{\leq 0}(A)\Ltens{A}D^{\leq 0}(A^{op})\subset D^{\leq m}(\mk).$ Let $M\in D_{\pspe}(A),$ and for convenience we assume that $M\in D^{\leq 0}(A).$ Choose some $t\in\Z,$ and a morphism $P\xto{f_t} A$ in $D(A\otimes A^{op}),$ such that $P\in D_{\perf}(A\otimes A^{op}),$ and $Cone(f_t)\in D^{<t}(A\otimes A^{op}).$ Then $M\Ltens{A}P\in D_{\perf}(A),$ and $Cone(M\Ltens{A}P\xto{\id\Ltens{A}f_t} M)\in D^{<(t+m)}(A).$ Since $t$ can be choosen arbitrarily, we conclude that $M$ is compactly approximable.\end{proof}

Let $X$ be a noetherian scheme. Let $\cE$ be a generator of $\Perf(X),$ and put $A_X:=\End(\cE).$ We have an equivalence $\Phi_X:D(X)\simeq D(A_X),$ inducing an equivalence $D_{\perf}(X)\simeq D_{\perf}(A_X).$

\begin{prop}\label{prop:Lipman_Neeman}Let $X,$ $A_X$ and $\Phi_X$ be as above.

1) The equivalence $\Phi_X$ identifies $D^-(\QCoh X)$ (resp. $D^+(\QCoh X),$ $D^b(\QCoh X)$) with $D^-(A_X),$  (resp. $D^+(A_X),$ $D^b(A_X)$).

2) $\Phi_X$ identifies $D^-_{coh}(X)$ (pseudo-coherent complexes) with the full subcategory of compactly approximable objects of $D(A_X).$\end{prop}

\begin{proof}Both assertions follow from \cite[Theorems 4.1, 4.2]{LN}.\end{proof}

\begin{prop}\label{prop:diagonal_geometric_proj_bounded_above} Let $X$ and $A_X$ be as above. Then $A_X\in D(A_X\otimes A_X^{op})$ is projectively bounded above and is compactly approximable.\end{prop}

\begin{proof}A generator $\cE\boxtimes \cE^{\vee}\in\Perf(X\times X)$ identifies $D(X\times X)$ with $D(A_X\otimes A_X^{op}).$ The structure sheaf  of the diagonal $\cO_{\Delta}\in D(X\times X)$ corresponds to the diagonal $A_X\mhyphen A_X$-bimodule. The assertion now follows from Proposition \ref{prop:Lipman_Neeman}.\end{proof}

\begin{cor}\label{cor:*_satisfied_for_proper_schemes}Let $X$  and $A_X$ be as above, and assume that $X$ is proper over $\mk.$ Then for any pseudo-perfect DG $A_X$-modules $M,$ $N$ we have an isomorphism $M\Ltens{A_X} N^*\xto{\sim}\bR\Hom_{A_X}(M,N)^*.$\end{cor}

\begin{proof}Since the DG algebra $A_X$ is proper, the assertion follows from Proposition \ref{prop:diagonal_geometric_proj_bounded_above} and Proposition \ref{prop:*_satisfied_for_all}.\end{proof}

For completeness we also prove the following result, which is certainly well-known to experts.

\begin{prop}\label{prop:pseudo_perfect_geometric} Let $X,$ $A_X$ and $\Phi_X$ be as above, and assume that $X$ is of finite type over $\mk.$ Then the functor $\Phi_X$ identifies 
$D^b_{coh,\propp}(X)$ with 
$D_{\pspe}(A_X).$\end{prop} 

\begin{proof}Clearly, if 
$\cF\in D^b_{coh,\propp}(X),$ 
then $\bR\Hom(\cE,\cF)=\bR\Gamma(\cE^{\vee}\Ltens{\cO_X}\cF)\in D_{\perf}(\mk),$ since $\supp(\cF)$ is proper. Thus, $\Phi_X(\cF)\in D_{\pspe}(A_X).$

Now suppose that $\cF\in D(X)$ is an object such that $\Phi_X(\cF)\in D_{\pspe}(A_X).$ We first show that $\cF\in D^b_{coh}(X).$ By Proposition \ref{prop:pseudo_perfect_comp_approx}, the object $\Phi_X(\cF)\in D(A_X)$ is compactly approximable, hence so is $\cF.$ By Proposition \ref{prop:Lipman_Neeman} 2), we obtain $\cF\in D^-_{coh}(X).$ Since $\Phi_X(\cF)\in D^b(A_X),$ we also conclude from Proposition \ref{prop:Lipman_Neeman} 1) that $\cF\in D^b(\QCoh X).$ Therefore, $\cF\in D^b_{coh}(X).$

It remains to show that $\supp(\cF)$ is proper. Assume the contrary. As in the proof of Proposition \ref{prop:pseudo-perfect_over_coherent}, we can find a closed embedding $i:C\to \supp(\cF),$ where $C$ is an affine integral curve. Take any perfect complex $\cG\in D_{\perf}(X)$ such that $\supp(\cG)=i(C).$ Then $\bR\Hom(\cG,\cF)\not\in D_{\perf}(\mk),$ a contradiction. This proves the proposition.\end{proof}


\begin{thebibliography}{GabRom}

\bibitem[Ab]{Ab}M.~Abouzaid, private communication.

\bibitem[AM]{AM}M.~F.~Atiyah, I.~G.~Macdonald, "Introduction to commutative algebra",
Addison-Wesley Publishing Co., Reading,
Mass.-London-Don Mills, Ont. 1969.

\bibitem[ACK]{ACK}E.~Arbarello, C.~De~Concini, V.~G.~Kac, "The infinite wedge representation and the reciprocity law for algebraic curves",
Theta functions, Proc. 35th Summer Res. Inst. Bowdoin Coll., Brunswick/ME 1987, Proc. Symp. Pure Math. 49, Pt. 1, 171-190 (1989).

\bibitem[Ca]{C}J.~W.~Calkin,
"Two-sided ideals and congruences in the ring of bounded operators in Hilbert space",
Ann. of Math. (2) 42, (1941). 839-873. 

\bibitem[CQ]{CQ}J.~Cuntz, D.~Quillen, "Algebra extensions and nonsingularity", J. Amer. Math. Soc. 8 (1995), no. 2, 251-289.

\bibitem[Dr1]{Dr}V.~Drinfeld, "Infinite-dimensional vector bundles in algebraic geometry: an introduction", The unity of mathematics, 263-304,
Progr. Math., 244, Birkh\"auser Boston, Boston, MA, 2006.

\bibitem[Dr2]{Dr2}V.~Drinfeld, "DG quotients of DG categories", J. Algebra 272 (2004), no. 2, 643-691.

\bibitem[F]{F}G.~Faltings, "Almost \'etale extensions", 
Cohomologies $p$-adiques et applications arithm\'etiques, II. 
Ast\'erisque No. 279 (2002), 185-270. 

\bibitem[GabRom]{GaRom}O.~Gabber, L.~Ramero, "Almost ring theory", 
Lecture Notes in Mathematics, 1800. Springer-Verlag, Berlin, 2003. vi+307 pp. ISBN: 3-540-40594-1.

\bibitem[GaiRoz]{GR}D.~Gaitsgory, N.~Rozenblyum, "A Study in Derived Algebraic Geometry: Volumes I and II", Mathematical Surveys and Monographs, vol. 221, American Mathematical Society, Providence, RI, 2017.

\bibitem[Hir]{H}H. Hironaka, "Resolution of singularities of an algebraic variety over a field of characteristic zero. I, II",
Ann. of Math. (2) 79 (1964), 109-203; ibid. (2) 79 1964 205-326.

\bibitem[Ho]{Ho}M.~Hovey, "Model categories", Mathematical surveys and monographs, Vol. 63, Amer.Math.
Soc., Providence 1998.

\bibitem[Ke1]{Ke}B.~Keller, "Deriving DG categories", Ann. Sci. \'Ecole Norm. Sup. (4) 27 (1994), no. 1, 63-102.

\bibitem[Ke2]{Ke2}B.~Keller, "On the cyclic homology category of exact categories", J. Pure Appl. Algebra 136 (1) (1999) 1-56.

\bibitem[KL]{KL}A.~Kuznetsov, V.~Lunts, "Categorical resolutions of irrational singularities", IMRN 2015, no. 13, 4536-4625.

\bibitem[LN]{LN}J.~Lipman, A.~Neeman. "Quasi-perfect scheme-maps and boundedness of the twisted inverse image functor",  Illinois J. Math. 51 (2007), no. 1, 209-236.

\bibitem[Lu]{Lu}V.~Lunts, "Categorical resolution of singularities", J. Algebra 323 (2010), no. 10, 2977-3003.

\bibitem[LO]{LO}V.~Lunts, D.~Orlov, "Uniqueness of enhancement for triangulated categories", J. Amer. Math. Soc. 23 (2010), no. 3, 853-908.

\bibitem[Nag]{Nag}M.~Nagata, "A generalization of the imbedding problem of an abstract variety in a complete variety", J. Math. Kyoto Univ. 3 1963 89-102.

\bibitem[Tab1]{Tab1} G.~Tabuada, "Th\'eorie homotopique des DG-cat\'egories", Th\'ese de L'Universit\'e Paris Diderot - Paris 7.

\bibitem[Tab2]{Tab2} G.~Tabuada, "Une structure de cat\'egorie de mod\'eles de Quillen sur la cat\'egorie des dg-cat\'egories", C. R. Acad. Sci. Paris Ser. I Math. 340 (1) (2005), 15-19.

\bibitem[Tam]{Tam}D.~Tamarkin, "What Do Dg-Categories Form?", Compositio Mathematica, vol. 143, no. 5, 2007, 1335-1358.

\bibitem[TT]{TT}D.~Tamarkin, B.~Tsygan, "Noncommutative differential calculus, homotopy BV algebras and formality conjectures", 	Methods Funct. Anal. Topology Vol. 6 (2000), no. 2 pp. 85-100.

\bibitem[Tate]{Ta}J.~Tate, "Residues of differentials on curves", Ann. Sci. \'Ecole Norm. Sup., 1 (1968), 149-159.

\bibitem[To]{T}B.~To\"en, "The homotopy theory of $dg$-categories and derived Morita theory",
Invent. Math. 167 (2007), no. 3, 615-667.

\bibitem[TV]{TV}B.~To\"en, M.~Vaqui\'e, "Moduli of objects in dg-categories", Ann. Sci. \'Ecole Norm. Sup. (4) 40 (2007), no. 3, 387-444.

\end{thebibliography}
\end{document}